\documentclass[12pt]{article}
\usepackage{amssymb,amsthm,hyperref,amsmath,bm,amsfonts}
\usepackage{arydshln}
\usepackage{verbatim}
\usepackage{graphicx}
\usepackage{float}

\usepackage{subfigure}
\usepackage{multirow}
\usepackage{color}
\usepackage{appendix}

\newtheorem{theorem}{Theorem}[section]

\newtheorem{prop}{Proposition}[section]
\newtheorem{definition}{Definition}[section]%

\newtheorem{remark}{Remark}[section]
\numberwithin{equation}{section}

\title{Coupling conditions for linear hyperbolic relaxation systems in two-scales problems}
\author{
Juntao Huang\thanks{Department of Mathematics, Michigan State University, Lansing 48824, USA. E-mail: huangj75@msu.edu}\qquad
Ruo Li\thanks{School of Mathematical Sciences, Peking University, Beijing 100871, China. E-mail: rli@math.pku.edu.cn} \qquad
Yizhou Zhou\thanks{Corresponding author, School of Mathematical Sciences, Peking University, Beijing 100871, China. E-mail: zhouyz@math.pku.edu.cn} 
}

\date{\today}
\begin{document}
\maketitle{}

\begin{abstract}
This work is concerned with coupling conditions for linear hyperbolic relaxation systems with multiple relaxation times.
In the region with small relaxation time, an equilibrium system can be used for computational efficiency. Under the assumption that the relaxation system satisfies the structural stability condition and the interface is non-characteristic, we derive a coupling condition at the interface to couple the two systems in a domain decomposition setting.
We prove the validity by the energy estimate and Laplace transform, which shows how the error of the domain decomposition method depends on the smaller relaxation time and the boundary layer effects.
In addition, we propose a discontinuous Galerkin (DG) scheme for solving the interface problem with the derived coupling condition and prove the $L^2$ stability.
We validate our analysis on the linearized Carleman model and the linearized Grad’s moment system and show the effectiveness of the DG scheme.
\end{abstract}

\hspace{-0.5cm}\textbf{Keywords:}
\small{hyperbolic relaxation systems, domain decomposition, characteristic boundaries, Kreiss condition, discontinuous Galerkin scheme}\\


\section{Introduction}
In this paper we are interested in the linear hyperbolic relaxation system with the multiple relaxation times:
\begin{equation}\label{problem}
U_t+ AU_{x} = \frac{1}{\epsilon(x)}QU,\qquad x \in \mathbb{R}, \quad t>0.
\end{equation}
Here the unknown function $U=U(x,t)\in \mathbb{R}^n$, $A\in \mathbb{R}^{n\times n}$ and $Q\in \mathbb{R}^{n\times n}$ are constant matrices, and $\epsilon(x)>0$ is the relaxation time which has different orders of magnitudes over the domain. In particular, we consider the case where $\epsilon(x)$ is a piecewise constant function:
\begin{equation*}
\epsilon(x)=\left\{{\begin{array}{*{20}l}
1,\qquad\qquad x\leq 0,\\[3mm]
\epsilon\ll 1,\qquad x>0.
\end{array}}\right.
\end{equation*}

This type of multiscale problems describe many important physical phenomena with different relaxation times such as the space shuttle reentry problems in aerodynamics \cite{BLPQ,Klar2}, the transport problems for materials with very different opacities \cite{BM,GJL}. In these problems, the relaxation time varies drastically near the interface.

For the problems with constant small relaxation time $\epsilon(x)\equiv \epsilon \ll 1$, a systematical theory has been developed in the framework of the structural stability condition in \cite{Y1}. Under this condition, without loss of generality, one may always assume that $A$ and $Q$ are symmetric and $Q=\text{diag}(0,S)$ with $S\in\mathbb{R}^{m\times m}$ and $S<0$ (see \cite{Y1,Y3} and the references cited therein). Corresponding to the partition of $Q$, we write
\begin{equation*}
U =\begin{pmatrix}
u\\[1mm]
v
\end{pmatrix},\qquad 
A = \begin{pmatrix}
A_{11} & A_{12} \\[1mm]
A_{12}^T & A_{22}
\end{pmatrix}.
\end{equation*}
It was proved that, as $\epsilon$ goes to zero, the $\epsilon$-dependent solution $U^\epsilon$ to \eqref{problem} converges to the solution of the corresponding equilibrium system \cite{Y1}:
\begin{align}\label{equilibrium}
\left\{
\begin{array}{l}
u_t+A_{11}u_{x} = 0,\\[3mm]
v=0.
\end{array}\right.
\end{align}

Motivated by this result, a natural idea to solve \eqref{problem} efficiently is the so-called domain decomposition (DD) method. Instead of solving the multiscales problem \eqref{problem} directly, the DD method couples two systems in different domains with suitable coupling conditions, see e.g. \cite{BLP,JLW}. In particular, the solution $U^l$ on the left domain $x\leq 0$ satisfies
\begin{equation}\label{leftproblem}
U^l_t + AU^l_{x} = QU^l,\qquad x\leq 0,
\end{equation}
while on the right domain $x>0$, the approximated solution $U^r$ is determined by
\begin{align}\label{rightproblem}
U^r=\begin{pmatrix}
u^r\\[1mm]
v^r
\end{pmatrix},\qquad
\left\{
\begin{array}{l}
u^r_t+A_{11}u^r_{x} = 0,\\[3mm]
v^r=0,
\end{array}\right.
\quad x>0. 
\end{align}
The two problems \eqref{leftproblem} and \eqref{rightproblem} are connected by the coupling conditions at $x=0$. This approach avoids the numerical difficulty caused by the stiff source and also reduces the computational cost by solving the equilibrium system with a smaller size on part of the domain. Clearly, the success of the DD method depends on the appropriate coupling conditions.

In this paper, our main goal is to develop a systematical approach to derive the coupling conditions for the general linear hyperbolic relaxation systems with the multiple relaxation times. Besides the structural stability stability, we also assume that the boundary $x=0$ is non-characteristic for \eqref{leftproblem}, i.e., the coefficient matrix $A$ is invertible. Note that the boundary may be characteristic for the equilibrium system \eqref{rightproblem}. By resorting to the theory of initial-boundary value problems (IBVPs) for hyperbolic relaxation systems \cite{ZY}, we derive the coupling conditions for \eqref{leftproblem} and \eqref{rightproblem} at the interface and prove the validity by the energy estimate and Laplace transform. With the derived coupling condition, we proceed to develop a numerical scheme in solving the interface problem. Since the discontinuous Galerkin (DG) method handles effectively the interactions across element boundaries due to the compact stencils, see e.g. the review paper \cite{cockburn2001runge}, we present a numerical algorithm in the DG framework and prove the $L^2$ stability for the semi-discrete schemes.

We briefly review the works which are devoted to derive coupling conditions for the two scales problems. In \cite{JLW}, the coupling condition is provided for the Jin-Xin model and the validity is proved for the linear case with constant coefficients. More theoretical results for nonlinear systems are presented in \cite{CJLW}. Here we would like to compare our work with \cite{JLW} and give two remarks: (1) Our work deals with the \emph{general} hyperbolic relaxation systems, which can be taken as a generalization of \cite{JLW} only focusing on the Jin-Xin model. Therefore, this could result in more potential in different applications. (2) The results of \cite{JLW} are obtained with the aid of the theory of IBVPs for Jin-Xin model developed in \cite{XX}. Similar with this routine, our derivation is based on the theory of IBVPs for general hyperbolic relaxation systems \cite{ZY}.

Similar to the topic of our work, a related issue is to prescribe the coupling conditions between the kinetic equation (e.g. the Boltzmann equation) and its hydrodynamic limit (e.g. the compressible Euler equation). Many works are devoted to this issue, see e.g. \cite{BLP,Besse,BLPQ,CLL,Dellacherie,GK,Klar}. One popular approach to couple the kinetic model and the fluid model is to solve a so-called Knudsen layer equation as a translation at the interface.
We would like to point out the relationship between this issue and our study on the coupling conditions for relaxation system \eqref{problem}. Two typical methods to solve the kinetic equation are the discrete velocity method \cite{Gatignol,Mieussens} and the moment closure method \cite{CFL,CFL2,Levermore}, in which the resultant models are hyperbolic relaxation systems. In this scenario, our strategy to derive the coupling conditions can be directly applied. We will show this by using two simple models in Section \ref{sec:numerical-example}.

The rest of the paper is organized as follows. In Section \ref{sec:coupling-condition}, we derive the coupling conditions and rigorously prove their validity. In Section \ref{sec:numerical-scheme}, we propose a DG scheme for solving the interface problem and prove the stability for the semi-discrete scheme. In Section \ref{sec:numerical-example}, we explicitly derive the coupling conditions for the linearized Carleman model and the Grad’s moment system. We validate our analysis and the effectiveness of the DG scheme numerically. 

\section{Coupling conditions}\label{sec:coupling-condition}

In this section, we firstly derive the continuity condition for the original problem \eqref{problem} at the interface. Then, based on this continuity condition, we present the coupling condition for the systems \eqref{leftproblem} and \eqref{rightproblem} in the framework of domain decomposition method. We prove its validity by the energy estimate and Laplace transform. 

\subsection{Continuity condition}

Firstly, we analyze the behavior of the original problem \eqref{problem} at the interface.
If $U=U(x,t)$ is a weak solution \cite{Dafermos,Smoller} to the original problem \eqref{problem}, we have 
\begin{equation}\label{eq:def-weak-solution}
\int_0^T\int_{\mathbb{R}}\left[ \phi_t^T U  + \phi_{x}^T AU  + \frac{1}{\epsilon(x)}\phi^T QU \right]dxdt = 0,
\end{equation}
for any test function $\phi=\phi(x,t)$ which has compact support contained in the domain $\{(x,t)\in\mathbb{R}^2 | t>0\}$.
Provided that the support of $\phi$ belongs to $D=D_-\cup D_+$ with $D_- = D\cap\{x<0\}$ and $D_+ = D\cap\{x>0\}$, \eqref{eq:def-weak-solution} can be written as
$$
\int_{D_-} \phi_t^T U  + \phi_{x}^T AU + \frac{1}{\epsilon(x)}\phi^T QU = - \int_{D_+} \phi_t^T U  + \phi_{x}^T AU + \frac{1}{\epsilon(x)}\phi^T QU.
$$
We assume that, in each $D_i$ with $i=+$ or $i=-$, the solution is smooth and thus satisfies the equation \eqref{problem}. Therefore, using integration by parts, we have
\begin{align*}
    &\int_{D_i} \left[\phi_t^T U  + \phi_{x}^T AU  + \frac{1}{\epsilon(x)} \phi^T QU \right] dxdt \\[2mm]
  = &\int_{D_i} \left[(\phi^T U)_t +  (\phi^T AU)_{x} - \phi^T U_t - \phi^T AU_{x} + \frac{1}{\epsilon(x)}\phi^T QU \right] dxdt\\[2mm]
  = &\int_{D_i} \left[(\phi^T U)_t + (\phi^T AU)_{x} \right] dxdt\\[2mm]
  = &-\int_{\partial D_i\cap \{x=0\}} \phi^T AU dt,
\end{align*}
and thus
$$
\int_{\partial D_{+}\cap \{x=0\}} \phi^T AU dt - \int_{\partial D_{-}\cap \{x=0\}} \phi^T AU dt = 0.
$$
Since $\phi$ is chosen arbitrarily, it follows that  
\begin{equation}\label{fluxcontinuity}
	\lim_{x\rightarrow 0^-}AU(x,t) = \lim_{x\rightarrow 0^+}AU(x,t).
\end{equation}
Thanks to the invertibility of $A$, we obtain the continuity condition
\begin{equation}\label{continuity}
	\lim_{x\rightarrow 0^-} U(x,t) = \lim_{x\rightarrow 0^+} U(x,t).
\end{equation}


\subsection{Coupling conditions}\label{section2.2}

In this part, we will derive the coupling conditions for the systems \eqref{leftproblem} and \eqref{rightproblem} from the continuity condition \eqref{continuity} with the aid of the theory of the  boundary conditions for the hyperbolic relaxation systems developed in \cite{ZY}.

For the convenience of readers, we firstly list the notations as follows:\\
\textbf{Numbers:}
\begin{itemize}
	\item $n\times n$: the size of $A$
	\item $m\times m$: the size of $S$
	\item $n^+/n^-$: the number of positive/negative eigenvalues of $A$
	\item $n_1^+/n_1^-/n_1^o$: the number of positive/negative/zero eigenvalues of $A_{11}$
\end{itemize}
\textbf{Characteristic decomposition:}
\begin{itemize}
	\item $\Lambda_+/\Lambda_-$: the diagonal matrix which represents the positive/negative eigenvalues of $A$. The size of $\Lambda_+$ is $n^+\times n^+$ and the size of $\Lambda_-$ is $n^-\times n^-$.

	\item $R = (R_+,R_-)$: the orthonormal eigenvectors of $A$ satisfying 
	\begin{equation}\label{eq:R+R-}
	R^TR=I_n,\qquad A(R_+,R_-)= (R_+,R_-)~\text{diag}(\Lambda_+,~\Lambda_-).
	\end{equation}
	The size of $R_+$ is $n\times n^+$ and the size of $R_-$ is $n\times n^-$.

	\item $\Lambda_{1+}/\Lambda_{1-}$: diagonal matrix which represents the positive/negative eigenvalues of $A_{11}$. The size of $\Lambda_{1+}$ is $n_1^+\times n_1^+$ and the size of $\Lambda_{1-}$ is $n_1^-\times n_1^-$.

	\item $P = (P_+,P_-,P_0)$: the orthonormal eigenvectors of $A_{11}$ satisfying 
	$$
	P^TP=I_{n-m},\qquad A_{11}(P_+,P_-,P_0)= (P_+,P_-,P_0)~\text{diag}(\Lambda_{1+},~\Lambda_{1-},~0).
	$$
	The sizes of $P_0$, $P_+$ and $P_-$ are $(n-m)\times n_1^o$, $(n-m)\times n_1^+$ and $(n-m)\times n_1^-$, respectively.
\end{itemize}

Clearly, it follows that $n=n^-+n^+$ and $n-m=n_1^++n_1^-+n_1^o$.
From the classical theory on IBVP for hyperbolic systems \cite{Benzoni,GKO}, we know that at $x=0$ the hyperbolic system \eqref{leftproblem} needs $n^-$ boundary conditions and \eqref{rightproblem} needs $n_1^+$ boundary conditions. 
In order to obtain such coupling conditions, we firstly try to construct an asymptotic solution to the original problem \eqref{problem}.

For $x>0$, the solution to \eqref{problem} allows boundary-layers when $\epsilon$ is sufficiently small. 
Motivated by the theory of boundary conditions for hyperbolic relaxation systems \cite{ZY}, we consider the following asymptotic solution to approximate the solution $U=U(x,t)$ to \eqref{problem} for $x>0$:
\begin{equation}\label{eq:boundary-layer-expansion}
\begin{pmatrix}
u^r\\[2mm]
v^r
\end{pmatrix}(x,t) + 
\begin{pmatrix}
\tilde{\mu} \\[2mm]
\tilde{\nu}
\end{pmatrix}(y,t) + \begin{pmatrix}
\hat{\mu} \\[2mm]
\hat{\nu}
\end{pmatrix}(z,t),\qquad y=\frac{x}{\epsilon},\quad z = \frac{x}{\sqrt{\epsilon}}.
\end{equation}
Here the outer solution $(u^r,v^r)$ satisfies the equation \eqref{rightproblem} while $(\hat{\mu},\hat{\nu})$ and $(\tilde{\mu},\tilde{\nu})$ are boundary-layer corrections. The continuity relation \eqref{continuity} indicates that
\begin{equation}\label{relation2.1}
U^l(0,t) = \begin{pmatrix}
u^r\\[2mm]
v^r
\end{pmatrix}(0,t) + 
\begin{pmatrix}
\tilde{\mu} \\[2mm]
\tilde{\nu}
\end{pmatrix}(0,t) + \begin{pmatrix}
\hat{\mu} \\[2mm]
\hat{\nu}
\end{pmatrix}(0,t).
\end{equation}
This  is our starting point to derive the coupling conditions for $U^l(0,t)$ and $u^r(0,t)$.


To determine the boundary-layer terms, we recall the theory developed in \cite{ZY}:
\begin{prop}[\cite{ZY}]\label{prop2.1}
The boundary-layer corrections in \eqref{eq:boundary-layer-expansion} are determined by
\begin{equation}\label{eq:boundary-layer-correction}
\begin{pmatrix}
\hat{\mu}\\[1mm]
\hat{\nu}
\end{pmatrix}=\begin{pmatrix}
P_0\\[1mm]
0
\end{pmatrix}\hat{w},\qquad
\begin{pmatrix}
\tilde{\mu}\\[1mm]
\tilde{\nu}
\end{pmatrix}=\begin{pmatrix}
N\\[1mm]
\tilde{K}
\end{pmatrix}\tilde{w},
\end{equation}
where $\hat{w}$ and $\tilde{w}$ satisfy the equations 
\begin{align*}
&\partial_t\hat{w} + \big(K^TS^{-1}K\big) \partial_{zz}\hat{w} = 0, \qquad \partial_y \tilde{w} = (\tilde{K}^TX\tilde{K})^{-1}(\tilde{K}^TS\tilde{K}) \tilde{w}.
\end{align*}
Here the matrices $K$, $\tilde{K}$, $X$ and $N$ are defined by
\begin{itemize}
	\item $K:=A_{12}^TP_0 \in\mathbb{R}^{m\times n_1^o}$
	\item $\tilde{K}\in \mathbb{R}^{m \times (m-n_1^o)}$: the orthogonal complement of $K$ satisfying $K^T\tilde{K}=0$
	\item $X := A_{22}-A_{12}^T \Big(P_+\Lambda_{1+}^{-1}P_+^T+P_-\Lambda_{1-}^{-1}P_-^T\Big)
	             A_{12} \in\mathbb{R}^{m\times m}$
	\item $N := -\Big(P_+\Lambda_{1+}^{-1}P_+^T+P_-\Lambda_{1-}^{-1}P_-^T\Big)A_{12}\tilde{K} $\\[1mm]
	$~~~~~~~~+ P_0 (K^TK)^{-1}\left[(K^TS\tilde{K})(\tilde{K}^TS\tilde{K})^{-1}(\tilde{K}^TX\tilde{K})-(K^TX\tilde{K})\right] \in\mathbb{R}^{(n-m)\times (m-n_1^o)}$.
\end{itemize}
\end{prop}
\noindent The detailed proof of this proposition can be found in \cite{ZY}. We omit the proof and only give some useful remarks.

\begin{remark}
Here the constant matrices $K$, $\tilde{K}$, $X$ and $N$ are uniquely determined and can be computed explicitly once the system \eqref{problem} is given. 
\end{remark}

\begin{remark}
If $A_{11}$ is invertible (i.e., the boundary is non-characteristic for the equilibrium system), there is no boundary-layer in the  scale of $x/\sqrt{\epsilon}$ and the equations for $(\tilde{\mu},\tilde{\nu})$ can be simplified as
$$
\tilde{\mu} = -A_{11}^{-1}A_{12}\tilde{\nu}, \qquad 
\partial_y \tilde{\nu} = (A_{22}-A_{12}^TA_{11}^{-1}A_{12})^{-1}S \tilde{\nu}.
$$
\end{remark}

\begin{remark}\label{rmk2.3}
It was proved in \cite{ZY} that $K^TS^{-1}K$ is negative definite and thereby $n_1^o$ boundary conditions should be prescribed for $\hat{w}$ at $z=0$. On the other hand, by the classical theory of ODEs, $\tilde{w}(0,t)$ should be given on the stable-manifold $\text{span}\{R_S\}$ where $R_S$ is the right-stable matrix (see the definition in Appendix \ref{appedix:GKC}) of $(\tilde{K}^TX\tilde{K})^{-1}(\tilde{K}^TS\tilde{K})$. From \cite{ZY}, we know that this matrix has $(n^+-n_1^o-n_1^+)$ negative eigenvalues and thereby the size of $R_S$ is $(m-n_1^o) \times (n^+-n_1^o-n_1^+)$.
\end{remark}

Thanks to Proposition \ref{prop2.1}, the boundary-layer corrections are uniquely determined by $\hat{w}$ and $\tilde{w}$, which satisfy a parabolic PDE and an ODE, respectively. From Remark \ref{rmk2.3}, we know that $\tilde{w}$ is on the stable-manifold $\text{span}\{R_S\}$ and thus it can be expressed as
\begin{equation}\label{eq:tilde-w}
\tilde{w}=R_S\tilde{w}^S.    
\end{equation}
Substituting \eqref{eq:boundary-layer-correction} and \eqref{eq:tilde-w} into the relation \eqref{relation2.1} yields
\begin{equation}\label{relation2.6}
U^l(0,t) = 
\begin{pmatrix}
u^r(0,t)\\[2mm]
0 
\end{pmatrix} +
\begin{pmatrix}
P_0\hat{w}(0,t) \\[2mm]
0
\end{pmatrix}
+ 
\begin{pmatrix}
NR_S \\[2mm]
\tilde{K}R_S
\end{pmatrix}\tilde{w}^S(0,t).
\end{equation}
By exploiting the characteristic decomposition, we can express $U^l(0,t)$ as the characteristic variables $\alpha_+$ and $\alpha_-$:
$$
U^l(0,t) = R_+ \alpha_+ + R_- \alpha_-,
$$
where $\alpha_+$ and $\alpha_-$ have $n^+$ and $n^-$ components, respectively.
Similarly, we introduce the characteristic decomposition for $u^r$:
$$
u^r(0,t) = P_+ \beta_{+}+ P_- \beta_{-} + P_0 \beta_0.
$$
Here $\beta_+$, $\beta_-$ and $\beta_0$ have $n_1^+$, $n_1^-$ and $n_1^o$ components, respectively.
Motivated by the classical theory on IBVP for hyperbolic systems \cite{GKO}, our goal is to express the incoming characteristic variable $\alpha_-$ (w.r.t. the left domain) and $\beta_+$ (w.r.t. the right domain), in terms of the knowns $\alpha_+$, $\beta_-$ and $\beta_0$.
Accordingly, we rewrite \eqref{relation2.6} as 
\begin{equation}\label{coco}
\begin{pmatrix}
-R_-, & 
\begin{pmatrix}
P_+ \\[2mm]
0
\end{pmatrix}, &\begin{pmatrix}
P_0 \\[2mm]
0
\end{pmatrix}, &
\begin{pmatrix}
NR_S \\[2mm]
\tilde{K}R_S
\end{pmatrix}
\end{pmatrix}
\begin{pmatrix}
\alpha_- \\[2mm]
\beta_+ \\[2mm]
\hat{w} \\[2mm]
\tilde{w}^S
\end{pmatrix} = R_+ \alpha_+ - \begin{pmatrix}
P_- \beta_- \\[2mm]
0
\end{pmatrix} - \begin{pmatrix}
P_0 \beta_0 \\[2mm]
0
\end{pmatrix}.
\end{equation}
This is an linear system which consists of $n$ equations and the number of unknowns is $n^-+n_1^++n_1^o+(n^+-n_1^o-n_1^+)=n$. 
The main result of this section is 
\begin{theorem}\label{coth}
Under the structural stability condition and the non-characteristic assumption $\det(A)\neq 0$, 
the linear system \eqref{coco} is uniquely solvable. Moreover, the coupling conditions for \eqref{leftproblem} and \eqref{rightproblem} can be expressed as
\begin{equation}\label{eq:coupling-condition-B}
\alpha_- = B_{l,l} \alpha_+ + B_{l,r} \beta_-, \qquad
\beta_+ = B_{r,r} \beta_- + B_{r,l} \alpha_+,
\end{equation}
Here $B_{l,l},B_{l,r},B_{r,r}$ and $B_{r,l}$ are constant matrices of size $n^-\times n^+$, $n^-\times n_1^+$, $n_1^+\times n_1^+$ and $n_1^+\times n^+$, respectively.
\end{theorem}

\begin{proof}
We will prove the theorem by resorting to the theory of boundary conditions for hyperbolic relaxation systems \cite{ZY}.  
In Appendix \ref{appedix:GKC}, we briefly review the  Generalized Kreiss condition (GKC) and show the following two conclusions: 1) the matrix $R_+^T$ satisfies the GKC; 2) the coefficient matrix on the left hand side of \eqref{coco} can be written as $(-R_-,~R_M^S(1,\infty))$ where 
$$
R_M^S(1,\infty)=
\begin{pmatrix}
P_+ & P_0 & NR_S\\[1mm]
0 & 0 & \tilde{K}R_S
\end{pmatrix}
$$
is the matrix $R_M^S(\xi,\eta)$ in the definition of GKC with $\xi=1$ and $\eta = \infty$ (See Appendix \ref{appedix:GKC} for details). 
The GKC implies that $R_+^TR_M^S(1,\infty)$ is invertible. Then the solvability of \eqref{coco} follows from the relation
\begin{equation*} 
\begin{pmatrix}
R_-^T \\[2mm]
R_+^T
\end{pmatrix} \Big(-R_-, ~R_M^S(1,\infty)\Big) = 
\begin{pmatrix}
-I~ & R_-^T R_M^S(1,\infty) \\[2mm]
~0~ & R_+^T R_M^S(1,\infty)
\end{pmatrix}.
\end{equation*}


Since the matrix $(-R_-,~R_M^S(1,\infty))$ is invertible, there exists a full-row rank matrix $B_L$ satisfying 
$$
B_L\begin{pmatrix}
-R_-, & 
\begin{pmatrix}
P_+ \\[2mm]
0
\end{pmatrix}
\end{pmatrix}~\text{is invertible and}~
B_L
\begin{pmatrix}
P_0 & NR_S\\[2mm]
0 & \tilde{K}R_S
\end{pmatrix}=0.
$$
Multiplying $B_L$ on the both sides of \eqref{coco}, we can eliminate $P_0\beta_0$ on the right hand side and thus express $(\alpha_-,\beta_+)$ in terms of $(\alpha_+,\beta_-)$. Namely, 
$$
\alpha_- = B_{l,l} \alpha_+ + B_{l,r} \beta_-, \qquad
\beta_+ = B_{r,r} \beta_- + B_{r,l} \alpha_+
$$
with $B_{l,l}, B_{l,r}, B_{r,r}, B_{r,l}$ constant matrices. 
\end{proof}

\begin{remark}
    The matrices $B_{l,l}, B_{l,r}, B_{r,r}, B_{r,l}$ in the coupling condition \eqref{eq:coupling-condition-B} could be explicitly computed once the hyperbolic relaxation system \eqref{problem} is given. We will show how to express the explicit coupling condition in two physical models in Section \ref{sec:numerical-example}.
\end{remark}

\begin{remark}
Note that the coupling condition \eqref{eq:coupling-condition-B} in Theorem \ref{coth} is independent of the zero characteristic variables $\beta_0$.
\end{remark}

\subsection{Validity}
In this part, we first prove the existence of the solution to the original problem \eqref{problem} and show the $L^2$ stability. Then, we show the validity of the derived coupling condition by the error estimate between the solution to the original problem and that to the interface problem \eqref{leftproblem}-\eqref{rightproblem}-\eqref{eq:coupling-condition-B}.

\begin{theorem}
Under the assumptions in Theorem \ref{coth} and the assumption that the initial value of \eqref{problem} satisfies $U_0\in H^1(\mathbb{R})$, there exists a solution $U^\epsilon=U^\epsilon(x,t)$ to the initial-value problem of \eqref{problem} which satisfies
\begin{align*}
\|U^\epsilon(\cdot,t)\|_{L^2(\mathbb{R})} \leq \|U^\epsilon(\cdot,0)\|_{L^2(\mathbb{R})}
\end{align*}
for any $\epsilon>0$ and $t>0$.
\end{theorem}
\begin{proof}
We first construct an initial-boundary value problem in the half plane $\{(x,t)\in\mathbb{R}^2 | x\ge0, t\ge0 \}$:
\begin{eqnarray*}
\left\{{\begin{array}{*{20}l}
\begin{pmatrix}
U_1 \\[2mm]
U_2
\end{pmatrix}_t + 
\begin{pmatrix}
-A & 0 \\[2mm]
0 & A
\end{pmatrix}
\begin{pmatrix}
U_1 \\[2mm]
U_2
\end{pmatrix}_x = 
\begin{pmatrix}
Q & 0 \\[2mm]
0 & Q/\epsilon
\end{pmatrix}
\begin{pmatrix}
U_1 \\[2mm]
U_2
\end{pmatrix},\qquad \\[8mm]
\begin{pmatrix}
U_1 \\[2mm]
U_2
\end{pmatrix}(x,0) = \begin{pmatrix}
U_0(-x) \\[2mm]
U_0(x)
\end{pmatrix},\\[8mm]
~U_1(0,t) - U_2(0,t)=0.
\end{array}}\right.
\end{eqnarray*}
It is easy to check that, for the above problem, the Kreiss condition \cite{GKO} 
$$
\det\left\{(I,-I)\begin{pmatrix}
R_- & 0\\[2mm]
0 & R_+
\end{pmatrix}\right\} = \det(R_-,-R_+)\neq 0
$$
holds where $\text{diag}(R_-,R_+)$ is the right-unstable matrix for the coefficient matrix $\text{diag}(-A,A)$.
Moreover, since $U_0\in H^1(\mathbb{R})$, it is continuous at $x=0$. It is easy to verify that the initial data are compatible with the boundary conditions, i.e.,
$$
U_1(0,0) = U_0(0) = U_2(0,0).
$$
According to the classical theory of hyperbolic systems \cite{Benzoni}, there exists a solution $(U_1,U_2)\in C([0,T];H^1(\mathbb{R}^+))\cap C^1([0,T];L^2(\mathbb{R}^+))$. Define the function $U^\epsilon = U^\epsilon(x,t)$:
$$
U^\epsilon(x,t) =\left\{ 
\begin{array}{l}
U_1(-x,t), \qquad x<0, \\[3mm]
U_2(x,t),~~\qquad x\ge0.
\end{array}\right.
$$
Then $U^\epsilon$ is a solution to the problem \eqref{problem} and also satisfies the continuity condition \eqref{continuity}. 

It remains to show the $L^2$ stability estimate. To do this, we multiply $(U^\epsilon)^T$ on both sides of  \eqref{problem}
and integrate over $x\in\mathbb{R}$ to obtain
\begin{align*}
&\frac{d}{dt}\|U^\epsilon(\cdot,t)\|_{L^2(\mathbb{R})}^2 + \int_{-\infty}^{0}[U_1^T(-x)AU_1(-x)]_x dx + \int_{0}^{+\infty} (U^T_2AU_2)_x dx \\[2mm]
\leq ~& \int_{-\infty}^{0} U_1^T(-x)QU_1(-x)dx + \frac{1}{\epsilon}\int_0^{+\infty} U_2^TQU_2dx.
\end{align*} 
By the continuity condition \eqref{continuity} and the non-positiveness of $Q$, we have 
\begin{align*}
&\frac{d}{dt}\|U^\epsilon(\cdot,t)\|_{L^2(\mathbb{R})}^2 \leq 0
\end{align*} 
and thereby obtain the estimate in the theorem.
\end{proof}

Next, we show the validity of the derived coupling condition by estimating the difference between the solution to the original problem \eqref{problem} and the solution to the coupling problem denoted by: 
\begin{equation}\label{eq:U-D}
U^D(x,t) = 
\left\{
\begin{array}{l}
U^l(x,t),\qquad x<0,\\[3mm]
U^r(x,t),\qquad x>0.
\end{array}\right.
\end{equation}
\begin{theorem}\label{thm2.3}
Suppose that the assumptions in Theorem \ref{coth} hold and the initial data to the original problem \eqref{problem} is in the equilibrium on $x>0$. Then the solution to \eqref{problem} and the solution to the coupling system \eqref{leftproblem}-\eqref{rightproblem}-\eqref{eq:coupling-condition-B} satisfy the following error estimate:
$$
\|U^\epsilon(\cdot,t)-U^D(\cdot,t)\|_{L^2(\mathbb{R})} \leq C(T) \epsilon^{1/4},\qquad t\in[0,T].
$$
\end{theorem}
\begin{proof}
\textbf{Step 1:}
Denote the approximate solution $U_\epsilon$ as 
$$
U_\epsilon = 
\left\{
\begin{array}{l}
U^l, \qquad x<0,\\[3mm]
\begin{pmatrix}
u^r\\[1mm]
v^r
\end{pmatrix}+\begin{pmatrix}
\hat{\mu}\\[1mm]
\hat{\nu}
\end{pmatrix}+\begin{pmatrix}
\tilde{\mu}\\[1mm]
\tilde{\nu}
\end{pmatrix}+\sqrt{\epsilon}\begin{pmatrix}
\hat{\mu}_1\\[1mm]
\hat{\nu}_1
\end{pmatrix}
,\qquad x>0.
\end{array}\right.
$$
Here $(u^r, v^r)$ satisfy \eqref{rightproblem}. The boundary-layer terms $(\hat{\mu},\hat{\nu})$ and $(\tilde{\mu},\tilde{\nu})$ are constructed in Proposition \ref{prop2.1}. Note that the high-order terms $(\hat{\mu}_1,\hat{\nu}_1)$ are included such that $U_\epsilon$ approximately satisfies the original problem \eqref{problem} for $x>0$. In fact, it was shown in \cite{ZY} that $(\hat{\mu}_1,\hat{\nu}_1)$ can be constructed properly such that $U_\epsilon$ satisfies
$$
\partial_tU_\epsilon + A\partial_xU_\epsilon = \frac{1}{\epsilon}QU_\epsilon + R_1 + 
\begin{pmatrix}
0\\[1mm]
R_2
\end{pmatrix},\qquad x>0,
$$
where the residuals $R_1$ and $R_2$ are bounded by
\begin{equation}\label{proof:residuals}
	\|R_1\|^2_{L^2([0,T]\times \mathbb{R}^+)} \leq C_1 \epsilon,\qquad \|R_2\|^2_{L^2([0,T]\times \mathbb{R}^+)} \leq C_2.
\end{equation}
Here $C_1$ and $C_2$ are two positive constants which are independent of $\epsilon$.
On the left domain $x<0$, it is clear that $U_\epsilon$ satisfies the equation \eqref{leftproblem}.
By the continuity condition \eqref{relation2.1}, we have
$$
\lim_{x\rightarrow 0^-} U_\epsilon(x,t)=\lim_{x\rightarrow 0^+} U_\epsilon(x,t) - \sqrt{\epsilon}\begin{pmatrix}
\hat{\mu}_1\\[1mm]
\hat{\nu}_1
\end{pmatrix}(0,t).
$$

Introduce the error functions $W_L=W_L(x,t)$ and $W_R=W_R(x,t)$ for $x>0$:
\begin{equation}\label{eq:error-W}
    W_L(x,t) = U^\epsilon(-x,t) - U_\epsilon(-x,t), \quad
    W_R(x,t) = U^\epsilon(x,t) - U_\epsilon(x,t).
\end{equation}
Since the initial data on $x>0$ is in the equilibrium, it is easy to verify that $W_L$ and $W_R$ satisfy
\begin{eqnarray}\label{error-IBVP}
\left\{{\begin{array}{*{20}l}
\partial_t\begin{pmatrix}
W_L \\[2mm]
W_R
\end{pmatrix} + 
\begin{pmatrix}
-A & 0 \\[2mm]
0 & A
\end{pmatrix}
\partial_x\begin{pmatrix}
W_L \\[2mm]
W_R
\end{pmatrix} = 
\begin{pmatrix}
Q & 0 \\[2mm]
0 & Q/\epsilon
\end{pmatrix}
\begin{pmatrix}
W_L \\[2mm]
W_R
\end{pmatrix}+\begin{pmatrix}
0 \\[2mm]
R
\end{pmatrix},\quad x>0,\\[7mm]
~W_L(0,t) - W_R(0,t) = -\sqrt{\epsilon}g(t),\\[3mm]
~(W_L, W_R)(x,0) = 0.
\end{array}}\right.
\end{eqnarray}
Here we have used the notation
$$
R := R_1 + 
\begin{pmatrix}
0\\[1mm]
R_2
\end{pmatrix},\qquad 
g(t) := \begin{pmatrix}
\hat{\mu}_1\\[1mm]
\hat{\nu}_1
\end{pmatrix}(0,t).
$$

To estimate the error $(W_L,W_R)$, we decompose \eqref{error-IBVP} into two problems
\begin{eqnarray}\label{error-eq1}
\left\{{\begin{array}{*{20}l}
\partial_t\begin{pmatrix}
W_{L,1} \\[2mm]
W_{R,1}
\end{pmatrix} + 
\begin{pmatrix}
-A & 0 \\[2mm]
0 & A
\end{pmatrix}
\partial_x\begin{pmatrix}
W_{L,1} \\[2mm]
W_{R,1}
\end{pmatrix} = 
\begin{pmatrix}
Q & 0 \\[2mm]
0 & Q/\epsilon
\end{pmatrix}
\begin{pmatrix}
W_{L,1} \\[2mm]
W_{R,1}
\end{pmatrix}+\begin{pmatrix}
0 \\[2mm]
R
\end{pmatrix},\quad x>0,\\[7mm]
~R_-^TW_{L,1}(0,t) = 0,\qquad R_+^TW_{R,1}(0,t)=0,\\[3mm]
~(W_{L,1}, W_{R,1})(x,0) = 0
\end{array}}\right.
\end{eqnarray}
and 
\begin{eqnarray}\label{error-eq2}
\left\{{\begin{array}{*{20}l}
\partial_t \begin{pmatrix}
W_{L,2} \\[2mm]
W_{R,2}
\end{pmatrix} + 
\begin{pmatrix}
-A & 0 \\[2mm]
0 & A
\end{pmatrix}
\partial_x \begin{pmatrix}
W_{L,2} \\[2mm]
W_{R,2}
\end{pmatrix} = 
\begin{pmatrix}
Q & 0 \\[2mm]
0 & Q/\epsilon
\end{pmatrix}
\begin{pmatrix}
W_{L,2} \\[2mm]
W_{R,2}
\end{pmatrix},\qquad x>0,\\[7mm]
~W_{L,2}(0,t) - W_{R,2}(0,t)=-\sqrt{\epsilon}g(t)-\left(W_{L,1}(0,t) - W_{R,1}(0,t)\right),\\[3mm]
~(W_{L,2}, W_{R,2})(x,0) = 0.
\end{array}}\right.
\end{eqnarray}
It is clear that $W_L=W_{L,1}+W_{L,2}$ and $W_R=W_{R,1}+W_{R,2}$ are the solutions to \eqref{error-IBVP}.

\textbf{Step 2:}
The error $(W_{L,1}, W_{R,1})$ in the first IBVP \eqref{error-eq1} is estimated by the energy method. Multiplying $\left(W_{L,1}^T,W_{R,1}^T\right)$ on both sides of the equation and integrating over $x>0$ yield
\begin{align*}
&\frac{d}{dt}\left(\big\|W_{L,1}\big\|_{L^2(\mathbb{R}^+)}^2+\big\|W_{R,1}\big\|_{L^2(\mathbb{R}^+)}^2\right) + \Big(W_{L,1}^T R_+ \Lambda_+  R_+^TW_{L,1} - W_{R,1}^T R_- \Lambda_-  R_-^TW_{R,1}\Big)\Big|_{x=0} \\[2mm]
\leq &~ 2 \int_0^{\infty} W_{L,1}^T Q W_{L,1} dx + \frac{2}{\epsilon} \int_0^{\infty} W_{R,1}^T Q W_{R,1} dx + 2\int_0^{\infty} W_{R,1}^T R_1 dx + 2 \int_0^{\infty}W_{R,1}^T \begin{pmatrix}
0\\[1mm]
R_2
\end{pmatrix}dx \\[2mm]
\leq &~ -\frac{2C_3}{\epsilon} \|W_{R,1}^{II}\|_{L^2(\mathbb{R}^+)}^2 
+ \|W_{R,1}\|_{L^2(\mathbb{R}^+)}^2 + \|R_1\|_{L^2(\mathbb{R}^+)}^2
+ \frac{C_3}{\epsilon} \|W_{R,1}^{II}\|_{L^2(\mathbb{R}^+)}^2 + \frac{4 \epsilon}{C_3}\|R_2\|_{L^2(\mathbb{R}^+)}^2\\[2mm]
\leq &~ \|W_{R,1}\|_{L^2(\mathbb{R}^+)}^2 + \|R_1\|_{L^2(\mathbb{R}^+)}^2
 + \frac{4 \epsilon}{C_3}\|R_2\|_{L^2(\mathbb{R}^+)}^2.
\end{align*}
Here $C_3>0$ is the minimum eigenvalue of $-S$ and independent of $\epsilon$. $W_{R,1}^I$ and $W_{R,1}^{II}$ denote the first $(n-m)$ components and the remaining $m$ components of $W_{R,1}$, respectively. Since the boundary term $(W_{L,1}^T R_+ \Lambda_+  R_+^TW_{L,1} - W_{R,1}^T R_- \Lambda_-  R_-^TW_{R,1})$ is nonnegative, we use the Gronwall's inequality and the relation \eqref{proof:residuals} to obtain
$$
\big\|W_{L,1}(\cdot,t)\big\|_{L^2(\mathbb{R}^+)}^2+\big\|W_{R,1}(\cdot,t)\big\|_{L^2(\mathbb{R}^+)}^2 \leq C(T) \epsilon.
$$
Moreover, by the boundary conditions we have
$$
|W_{L,1}(0,t)|^2=|R_+^TW_{L,1}(0,t)|^2 \leq C_4 W_{L,1}^T R_+ \Lambda_+  R_+^TW_{L,1}
$$
and
$$
|W_{R,1}(0,t)|^2=|R_-^TW_{R,1}(0,t)|^2 \leq C_4 W_{R,1}^T R_- (-\Lambda_-)  R_-^TW_{R,1}
$$
with $C_4>0$. This implies the estimate for boundary terms 
$$
\int_0^T |W_{L,1}(0,t)|^2 + |W_{R,1}(0,t)|^2 dt \leq C(T) \epsilon.
$$

\textbf{Step 3:}
To obtain the estimate for the second IBVP \eqref{error-eq2}, we check the GKC for the boundary condition in \eqref{error-eq2}, i.e.,
$$
|\det\{(I,-I)R_M^S(\xi,\eta)\}|\geq c_K
$$
where $R_M^S(\xi,\eta)$ is the right-stable matrix for 
$$
M(\xi,\eta)=
\begin{pmatrix}
-A & 0 \\[1mm]
0 & A
\end{pmatrix}^{-1}\left[\eta\begin{pmatrix}
0 & 0 \\[1mm]
0 & Q
\end{pmatrix}-\xi \begin{pmatrix}
I & 0 \\[1mm]
0 & I
\end{pmatrix}\right].
$$
It is easy to see that $R_M^S(\xi,\eta)=\text{diag}(R_{1}^S(\xi),R_{2}^S(\xi,\eta))$ with $R_1^S(\xi)$ the right-stable matrix for $\xi A^{-1}$ and $R_2^S(\xi,\eta)$ the right-stable matrix for $A^{-1}(\eta Q-\xi I)$. Clearly, $R_{1}^S(\xi)$ can be chosen as $R_-$ and it suffices to check 
$$
|\det\{(I,-I)R_M^S(\xi,\eta)\}|=|\det(R_-,-R_2^S(\xi,\eta))|\geq c_K.
$$
In fact, the above inequality follows from the fact $|\det\{R_+^TR_2^S(\xi,\eta)\}|\geq c_K$ (see Appendix \ref{appedix:GKC}) and the relation
$$
\begin{pmatrix}
R_-^T \\[2mm]
R_+^T
\end{pmatrix}(R_-,-R_2^S(\xi,\eta))=\begin{pmatrix}
I & -R_-^TR_2^S(\xi,\eta)\\[2mm]
0 & R_+^TR_2^S(\xi,\eta)
\end{pmatrix}.
$$
Thanks to Proposition \ref{prop:appendix-estimate} in Appendix \ref{appedix:GKC}, we have the following estimate:
\begin{align*}
\big\|W_{L,2}(\cdot,t)\big\|_{L^2(\mathbb{R}^+)}^2+\big\|W_{R,2}(\cdot,t)\big\|_{L^2(\mathbb{R}^+)}^2 \leq & ~C(T) \int_0^T |W_{L,2}(0,t)|^2 + |W_{R,2}(0,t)|^2 dt \\[2mm]
\leq & ~C(T) \int_0^T \epsilon|g(t)|^2+|W_{L,1}(0,t)|^2 + |W_{R,1}(0,t)|^2 dt\\[2mm]
\leq & ~C(T) \epsilon.
\end{align*}

\textbf{Step 4:}
Combining the estimates for $(W_{L,1},W_{R,1})$ and $(W_{L,2},W_{R,2})$, we know that $\|W_L(\cdot,t)\|_{L^2(\mathbb{R}^+)}^2 + \|W_R(\cdot,t)\|_{L^2(\mathbb{R}^+)}^2 \leq C(T)\epsilon$. Recall the definition of $W_L$ and $W_R$ in \eqref{eq:error-W}, we have the estimate:
$$
\|(U^\epsilon-U_\epsilon)(\cdot,t)\|_{L^2(\mathbb{R})}\leq C(T) \sqrt{\epsilon}.
$$
At last, from the relations 
$$
\int_{0}^{+\infty}|(\hat{\mu},\hat{\nu})(\frac{x}{\sqrt{\epsilon}},t)|^2 dx 
= \sqrt{\epsilon}\int_{0}^{+\infty}|(\hat{\mu},\hat{\nu})(y,t)|^2 dy \leq C\sqrt{\epsilon}
$$
and 
$$
\int_{0}^{+\infty}|(\tilde{\mu},\tilde{\nu})(\frac{x}{\epsilon},t)|^2 dx 
= \epsilon\int_{0}^{+\infty}|(\hat{\mu},\hat{\nu})(z,t)|^2 dz \leq C \epsilon,
$$
we know that $\|(U_\epsilon-U^D)(\cdot,t)\|_{L^2(\mathbb{R})}\leq C \epsilon^{1/4}$. Consequently, we obtain the error estimate from the triangle inequality 
$$
\|(U^\epsilon-U^D)(\cdot,t)\|_{L^2(\mathbb{R})}\leq \|(U^\epsilon-U_\epsilon)(\cdot,t)\|_{L^2(\mathbb{R})} + \|(U_\epsilon-U^D)(\cdot,t)\|_{L^2(\mathbb{R})} \leq C \epsilon^{1/4}.
$$
\end{proof}

\begin{remark}\label{rmk25}
From the Step 4 in the proof, we know that the error of order $O(\epsilon^{1/4})$ is due to the boundary-layer terms of $\sqrt{\epsilon}$ scale. Denote $U^\epsilon_i$ and $U_{i}^D$ as the $i$-th component of $U^\epsilon$ and $U^D$. If there is no $\sqrt{\epsilon}$ scale boundary-layer for $U^\epsilon_i$, then the error estimate for this variable should be 
$$
\|(U_i^\epsilon-U_{i}^D)(\cdot,t)\|_{L^2(\mathbb{R})}\leq C \epsilon^{1/2}.
$$ 
We will verify it numerically in Section \ref{sec:numerical-example}.
\end{remark}

\section{Numerical scheme}\label{sec:numerical-scheme}
With the derived coupling condition \eqref{eq:coupling-condition-B}, we proceed to present a DG scheme to solve the interface problem \eqref{leftproblem}-\eqref{rightproblem}-\eqref{eq:coupling-condition-B}. We also show the weighted $L^2$ stability of the semi-discrete DG scheme.

For the sake of simplicity, we only consider compactly supported boundary conditions, i.e., the solution has compact support in $I=(I_l, I_r)\subset\mathbb{R}$ with $I_l < 0 < I_r$ for some finite time. We introduce the usual notation of the DG method. Let us start by assuming the following mesh to cover the interval $I=(I_l, I_r)$, consisting of cells $I_i=(x_{i-\frac{1}{2}},x_{i+\frac{1}{2}})$ for $i=\pm N,\pm(N-1),...,\pm 1,0$ where
$$
I_l = x_{-N-\frac{1}{2}} < x_{-N+\frac{1}{2}} < \cdots < x_{\frac{1}{2}} < x_{\frac{3}{2}} < \cdots < x_{N-\frac{1}{2}} < x_{N+\frac{1}{2}} = I_r.
$$
We assume that the interface is located at $x_{\frac{1}{2}}=0$. 
Note that in the interface problem, the solution $U^l$ on the left domain $I^- := I\cap\mathbb{R}^-$ is a vector-valued function of size $n$, which is different from the size $(n-m)$ for the solution $u^r$ on the right domain $I^+ := I\cap\mathbb{R}^+$. Therefore, we define two piecewise polynomial spaces as the DG finite element spaces on $I^-$ and $I^+$, respectively:
$$
\mathbb{V}_h^{-} = \{ v\in [L^2(I^-)]^n: v|_{I_i}\in [P^k(I_i)]^n, ~ -N\leq i \leq 0 \}
$$
and
$$
\mathbb{V}_h^{+} = \{ v\in [L^2(I^+)]^{n-m}: v|_{I_i}\in [P^k(I_i)]^{n-m}, ~ 1\leq i \leq N \}
$$
where $P^k(I_i)$ denotes the set of polynomials of degree up to $k\ge1$ defined on the cell $I_i$. Note that each component of a vector-valued  function in $\mathbb{V}_h^{-}$ and $\mathbb{V}_h^{+}$ is allowed to have discontinuities across element interfaces.

Next, we present the DG scheme for solving the interface problem. On the left domain $I^-$, the equation \eqref{leftproblem} is solved by the  scheme: find $U_h\in \mathbb{V}_h^{-}$ such that, for any $V_h\in \mathbb{V}_h^{-}$ and $-N\leq i\leq 0$, there holds the following relation:
\begin{equation}\label{DGscheme}
\int_{I_i} V_h^T (\partial_tU_h)  - \int_{I_i} (\partial_xV_h)^T A U_h  +V_h^T(x^-_{i+\frac{1}{2}})\hat{F}_{i+\frac{1}{2}} - V_h^T(x^+_{i-\frac{1}{2}})\hat{F}_{i-\frac{1}{2}} = \int_{I_i}V_h^TQU_h.
\end{equation}
Here $\hat{F}_{i+\frac{1}{2}}$ is taken as the upwind numerical flux \cite{leveque1992numerical}
\begin{equation}\label{eq:flux-upwind}
\hat{F}_{i+\frac{1}{2}} = A^+U_h(x^-_{i+\frac{1}{2}})+A^-U_h(x^+_{i+\frac{1}{2}})
\end{equation}
with $A^+=R_+\Lambda_+R_+^T$ and $A^-=R_-\Lambda_-R_-^T$.

Similarly, on the right domain $I^+$, the DG scheme reads as: find $u_h\in \mathbb{V}_h^{+}$ such that, for any $v_h\in \mathbb{V}_h^{+}$ and $1\leq i\leq N$
\begin{align}\label{DGschemeeq}
\int_{I_i} v_h^T (\partial_tu_h)  - \int_{I_i} (\partial_xv_h)^T A_{11} u_h  +v_h^T(x^-_{i+\frac{1}{2}})\hat{f}_{i+\frac{1}{2}} - v_h^T(x^+_{i-\frac{1}{2}})\hat{f}_{i-\frac{1}{2}} = 0.
\end{align}
Here $\hat{f}_{i+\frac{1}{2}}$ is the upwind numerical flux
\begin{equation}\label{eq:flux-upwind-f}
\hat{f}_{i+\frac{1}{2}} = A_{11}^+u_h(x^-_{i+\frac{1}{2}})+A_{11}^-u_h(x^+_{i+\frac{1}{2}}),
\end{equation}
with $A_{11}^+=P_+\Lambda_{1+}P_+^T$ and $A_{11}^-=P_-\Lambda_{1-}P_-^T$.

At the interface $x_{\frac{1}{2}} = 0$, the value of $R_-^TU_h(x^+_{\frac{1}{2}})$ and $P_+^Tu_h(x^-_{\frac{1}{2}})$ are determined by the coupling condition \eqref{eq:coupling-condition-B}:
\begin{equation}\label{DGBC}
\left\{
\begin{array}{l}
R_-^TU_h(x^+_{\frac{1}{2}}) = B_{l,l}\Big(R_+^TU_h(x^-_{\frac{1}{2}})\Big) + B_{l,r}\Big(P_-^Tu_h(x^+_{\frac{1}{2}})\Big)\\[5mm]
P_+^Tu_h(x^-_{\frac{1}{2}}) = B_{r,r}\Big(P_-^Tu_h(x^+_{\frac{1}{2}})\Big) + B_{r,l}\Big(R_+^TU_h(x^-_{\frac{1}{2}})\Big).
\end{array}\right.
\end{equation}

It is easy to see that the scheme \eqref{DGscheme}-\eqref{DGschemeeq} together with the treatment \eqref{DGBC} at the interface give a complete numerical algorithm.
In the following theorem, we will show that the semi-discrete scheme satisfies the weighted $L^2$ stability.
\begin{theorem}
Suppose that the structural stability condition and the non-characteristic assumption $\det(A)\neq 0$ hold.
The semi-discrete DG scheme \eqref{DGscheme}-\eqref{DGschemeeq} together with the coupling condition \eqref{DGBC} satisfy the following stability estimate
$$
\|U_h(t)\|^2_{\delta_-}+\|u_h(t)\|^2_{\delta_+}\leq C(T,\delta) \left(\|U_h(0)\|^2_{\delta_-}+\|u_h(0)\|^2_{\delta_+}\right),
$$
for any $t\in[0,T]$ with $C=C(T,\delta)>0$. Here the weighted $L^2$-norms $\|\cdot\|_{\delta_-}$ and $\|\cdot\|_{\delta_+}$ are defined by
$$
\|U\|^2_{\delta_-}=\int_{\mathbb{R}^-}|R_+^TU|^2+\delta|R_-^TU|^2 dx,\quad 
\|u\|^2_{\delta_+}=\int_{\mathbb{R}^+}\delta|P_+^Tu|^2+|P_0^Tu|^2+|P_-^Tu|^2 dx,
$$
where $\delta>0$ is a constant.
\end{theorem}
\begin{proof}
\textbf{Step 1:} In \eqref{DGscheme}, by taking $V_h^T=U_h^T( R_+R_+^T+ \delta R_-R_-^T)$ with $\delta>0$ to be specified later, we have 
\begin{align*} 
&~\frac{1}{2}\frac{d}{dt}\int_{I_i}|R_+^TU_h|^2+\delta|R_-^TU_h|^2 dx \\[2mm]
=&~ \frac{1}{2} \left( U_h^T A^+ U_h + \delta U_h^T A^- U_h\right)\Big|_{x^+_{i-\frac{1}{2}}}^{x^-_{i+\frac{1}{2}}} ~+ \int_{I_i}U_h^T( R_+R_+^T+\delta R_-R_-^T)QU_h dx \\[3mm]
  &- U_h^T(x^-_{i+\frac{1}{2}})( R_+R_+^T+\delta R_-R_-^T)\hat{F}_{i+\frac{1}{2}} + U_h^T(x^+_{i-\frac{1}{2}})( R_+R_+^T+\delta R_-R_-^T)\hat{F}_{i-\frac{1}{2}}.
\end{align*}
According to the definition of numerical flux $\hat{F}_{i+\frac{1}{2}}$ in \eqref{eq:flux-upwind}, it is easy to compute 
$$
(R_+R_+^T+\delta R_-R_-^T)\hat{F}_{i+\frac{1}{2}}=A^+U_h(x^-_{i+\frac{1}{2}})+ \delta A^-U_h(x^+_{i+\frac{1}{2}}).
$$
Denote $U_{i+\frac{1}{2}}^-=U_h(x^-_{i+\frac{1}{2}})$ and $U_{i+\frac{1}{2}}^+=U_h(x^+_{i+\frac{1}{2}})$ for each $i$.
Then the above equation can be written as
\begin{align} 
 &~\frac{1}{2}\frac{d}{dt}\int_{I_i}|R_+^TU_h|^2+\delta|R_-^TU_h|^2 dx - \int_{I_i}U_h^T( R_+R_+^T+\delta R_-R_-^T)QU_h dx \nonumber\\[3mm]
=&~ J_1(U_{i+\frac{1}{2}}^-,U_{i+\frac{1}{2}}^+) - J_2(U_{i-\frac{1}{2}}^-,U_{i-\frac{1}{2}}^+), \label{eq:theorem3.1-1}
\end{align}
where 
\begin{align*}
&J_1(U_{i+\frac{1}{2}}^-,U_{i+\frac{1}{2}}^+) = \frac{1}{2} \Big[ (U_{i+\frac{1}{2}}^-)^T A^+ U_{i+\frac{1}{2}}^- + \delta (U_{i+\frac{1}{2}}^-)^T A^- U_{i+\frac{1}{2}}^-\Big] - (U_{i+\frac{1}{2}}^-)^T \Big[A^+U_{i+\frac{1}{2}}^- + \delta A^- U_{i+\frac{1}{2}}^+\Big] \\[3mm]
&J_2(U_{i-\frac{1}{2}}^-,U_{i-\frac{1}{2}}^+) = \frac{1}{2} \Big[ (U_{i-\frac{1}{2}}^+)^T A^+ U_{i-\frac{1}{2}}^+ + \delta (U_{i-\frac{1}{2}}^+)^T A^- U_{i-\frac{1}{2}}^+\Big] - (U_{i-\frac{1}{2}}^+)^T\Big[A^+U_{i-\frac{1}{2}}^- + \delta A^- U_{i-\frac{1}{2}}^+\Big].
\end{align*}
By a simple computation, we have
\begin{align*}
&J_2(U_{i-\frac{1}{2}}^-,U_{i-\frac{1}{2}}^+) - J_1(U_{i-\frac{1}{2}}^-,U_{i-\frac{1}{2}}^+)\\[2mm]
= & ~\frac{1}{2} \Big[ (U_{i-\frac{1}{2}}^+)^T A^+ U_{i-\frac{1}{2}}^+ + \delta (U_{i-\frac{1}{2}}^+)^T A^- U_{i-\frac{1}{2}}^+\Big] - (U_{i-\frac{1}{2}}^+)^T\Big[A^+U_{i-\frac{1}{2}}^- + \delta A^- U_{i-\frac{1}{2}}^+\Big] \\[2mm]
 &- \frac{1}{2} \Big[ (U_{i-\frac{1}{2}}^-)^T A^+ U_{i-\frac{1}{2}}^- + \delta (U_{i-\frac{1}{2}}^-)^T A^- U_{i-\frac{1}{2}}^-\Big] + (U_{i-\frac{1}{2}}^-)^T \Big[A^+U_{i-\frac{1}{2}}^- + \delta A^- U_{i-\frac{1}{2}}^+\Big] \\[2mm]
=&~\frac{1}{2} (U_{i-\frac{1}{2}}^+-U_{i-\frac{1}{2}}^-)^T A^+ (U_{i-\frac{1}{2}}^+-U_{i-\frac{1}{2}}^-)
-\frac{\delta}{2} (U_{i-\frac{1}{2}}^+-U_{i-\frac{1}{2}}^-)^T A^- (U_{i-\frac{1}{2}}^+-U_{i-\frac{1}{2}}^-) \ge 0.
\end{align*}
Thereby \eqref{eq:theorem3.1-1} becomes 
\begin{align} 
&\frac{1}{2}\frac{d}{dt}\int_{I_i}U_h^T( R_+R_+^T+\delta R_-R_-^T)U_h dx 
-\int_{I_i} U_h^T(R_+R_+^T+\delta R_-R_-^T)QU_h dx \nonumber\\[2mm]
=&~ J_1(U_{i+\frac{1}{2}}^-,U_{i+\frac{1}{2}}^+) - J_1(U_{i-\frac{1}{2}}^-,U_{i-\frac{1}{2}}^+) - \big[J_2(U_{i-\frac{1}{2}}^-,U_{i-\frac{1}{2}}^+)-J_1(U_{i-\frac{1}{2}}^-,U_{i-\frac{1}{2}}^+)\big] \nonumber \\[3mm]
\leq &~J_1(U_{i+\frac{1}{2}}^-,U_{i+\frac{1}{2}}^+) - J_1(U_{i-\frac{1}{2}}^-,U_{i-\frac{1}{2}}^+). \label{eq:theorem3.1-2}
\end{align}
On the other hand, there exists a constant $C_\delta>0$ such that 
\begin{align*} 
\int_{I_i}U_h^T( R_+R_+^T+\delta R_-R_-^T)QU_h dx \leq C_\delta \int_{I_i}U_h^T( R_+R_+^T+\delta R_-R_-^T)U_h dx.
\end{align*}
Making summation over $i$ from $-N$ to $0$ in  \eqref{eq:theorem3.1-2}, we have
\begin{align}\label{eq:theorem3.1-step1}
\frac{d}{dt}\|U_h\|^2_{\delta_-}  \leq & ~2J_1(U_{\frac{1}{2}}^-,U_{\frac{1}{2}}^+) + 2C_\delta \|U_h\|^2_{\delta_-}.
\end{align}
Here, we use the compact supported boundary condition to obtain $J_1(U_{-N-\frac{1}{2}}^-,U_{-N-\frac{1}{2}}^+) = 0$.

\textbf{Step 2:}
Similar to Step 1, we analyze the scheme for the equilibrium system. Taking $v_h^T=u_h^T( \delta P_+P_+^T + P_0P_0^T + P_-P_-^T)$ in \eqref{DGschemeeq}, we have 
\begin{align} 
&~\frac{1}{2}\frac{d}{dt}\int_{I_i}\delta|P_+^Tu_h|^2+|P_0^Tu_h|^2+|P_-^Tu_h|^2 dx \nonumber \\[2mm]
= &~ \frac{1}{2} \left( \delta u_h^T A_{11}^+ u_h + u_h^T A_{11}^- u_h\right)\Big|_{x^+_{i-\frac{1}{2}}}^{x^-_{i+\frac{1}{2}}} - u_h^T(x^-_{i+\frac{1}{2}})( \delta P_+P_+^T + P_0P_0^T + P_-P_-^T)\hat{f}_{i+\frac{1}{2}}\nonumber \\[2mm]
  & + u_h^T(x^+_{i-\frac{1}{2}})( \delta P_+P_+^T + P_0P_0^T + P_-P_-^T)\hat{f}_{i-\frac{1}{2}}. \label{eq:theorem3.1-3}
\end{align}
According to the definition of $\hat{f}_{i+\frac{1}{2}}$ in \eqref{eq:flux-upwind-f}, it follows that
\begin{align*} 
( \delta P_+P_+^T + P_0P_0^T + P_-P_-^T)\hat{f}_{i+\frac{1}{2}} =&~ ( \delta P_+P_+^T + P_0P_0^T + P_-P_-^T)\big[A_{11}^+u_h(x^-_{i+\frac{1}{2}})+A_{11}^-u_h(x^+_{i+\frac{1}{2}})\big]\\[2mm]
=&~ \delta A_{11}^+u_h(x^-_{i+\frac{1}{2}}) + A_{11}^-u_h(x^+_{i+\frac{1}{2}}).
\end{align*}
Denote $u_{i+\frac{1}{2}}^-=u_h(x^-_{i+\frac{1}{2}})$ and $u_{i+\frac{1}{2}}^+=u_h(x^+_{i+\frac{1}{2}})$. The equation \eqref{eq:theorem3.1-3} becomes 
\begin{align}\label{eq:theorem3.1-4}
&~\frac{1}{2}\frac{d}{dt}\int_{I_i}\delta|P_+^Tu_h|^2+|P_0^Tu_h|^2+|P_-^Tu_h|^2 dx  = j_1(u_{i+\frac{1}{2}}^-,u_{i+\frac{1}{2}}^+)-j_2(u_{i-\frac{1}{2}}^-,u_{i-\frac{1}{2}}^+),
\end{align}
where
\begin{align*}
&j_1(u^-_{i+\frac{1}{2}},u^+_{i+\frac{1}{2}}) = \frac{1}{2} \Big[ \delta (u_{i+\frac{1}{2}}^-)^T A_{11}^+ u_{i+\frac{1}{2}}^- + (u_{i+\frac{1}{2}}^-)^T A_{11}^- u_{i+\frac{1}{2}}^-\Big] - (u_{i+\frac{1}{2}}^-)^T\Big[\delta A_{11}^+ u_{i+\frac{1}{2}}^- + A_{11}^- u_{i+\frac{1}{2}}^+\Big] \\[3mm]
&j_2(u^-_{i-\frac{1}{2}},u^+_{i-\frac{1}{2}}) = \frac{1}{2} \Big[ \delta (u_{i-\frac{1}{2}}^+)^T A_{11}^+ u_{i-\frac{1}{2}}^+ + (u_{i-\frac{1}{2}}^+)^T A_{11}^- u_{i-\frac{1}{2}}^+\Big] - (u_{i-\frac{1}{2}}^+)^T\Big[\delta A_{11}^+ u_{i-\frac{1}{2}}^- + A_{11}^- u_{i-\frac{1}{2}}^+\Big].
\end{align*}
With the same computation as previous, it can be shown that 
\begin{align*}
&~j_1(u^-_{i+\frac{1}{2}},u^+_{i+\frac{1}{2}}) - j_2(u^-_{i+\frac{1}{2}},u^+_{i+\frac{1}{2}}) \\[2mm]
=&~ \frac{1}{2}(u_{i+\frac{1}{2}}^+-u_{i+\frac{1}{2}}^-)^T A_{11}^- (u_{i+\frac{1}{2}}^+-u_{i+\frac{1}{2}}^-) - \frac{\delta}{2}(u_{i+\frac{1}{2}}^+-u_{i+\frac{1}{2}}^-)^T A_{11}^+ (u_{i+\frac{1}{2}}^+-u_{i+\frac{1}{2}}^-)\leq 0.
\end{align*}
Then \eqref{eq:theorem3.1-3} gives 
\begin{align*}
&~\frac{1}{2}\frac{d}{dt}\int_{I_i}\delta|P_+^Tu_h|^2+|P_0^Tu_h|^2+|P_-^Tu_h|^2 dx  \\[2mm]
=&~ j_2(u_{i+\frac{1}{2}}^-,u_{i+\frac{1}{2}}^+)-j_2(u^-_{i-\frac{1}{2}},u^+_{i-\frac{1}{2}}) + [j_1(u_{i+\frac{1}{2}}^-,u_{i+\frac{1}{2}}^+) - j_2(u_{i+\frac{1}{2}}^-,u_{i+\frac{1}{2}}^+)] \\[3mm]
\leq &~j_2(u_{i+\frac{1}{2}}^-,u_{i+\frac{1}{2}}^+)-j_2(u^-_{i-\frac{1}{2}},u^+_{i-\frac{1}{2}}).
\end{align*}
Making summation over $i$ from $1$ to $N$, we have
\begin{align}\label{eq:theorem3.1-step2} 
\frac{d}{dt}\|u_h\|^2_{\delta_+}  \leq -2 j_2(u^-_{\frac{1}{2}},u^+_{\frac{1}{2}}).
\end{align}
Here, we use the compact supported boundary condition to obtain $j_2(u_{N+\frac{1}{2}}^-, u_{N+\frac{1}{2}}^+) = 0$.

\textbf{Step 3:} In this step, we estimate the boundary terms on the right-hand sides of \eqref{eq:theorem3.1-step1} and \eqref{eq:theorem3.1-step2}.
Recall that
\begin{align*} 
2J_1(U_{\frac{1}{2}}^-,U_{\frac{1}{2}}^+)
 = &~ \big[ (U_{\frac{1}{2}}^-)^T A^+ U_{\frac{1}{2}}^- + \delta (U_{\frac{1}{2}}^-)^T A^- U_{\frac{1}{2}}^-\big] - 2(U_{\frac{1}{2}}^-)^T \big[A^+U_{\frac{1}{2}}^- + \delta A^- U_{\frac{1}{2}}^+\big] \\[2mm]
 = & -(U_{\frac{1}{2}}^-)^T A^+ U_{\frac{1}{2}}^- + \delta (U_{\frac{1}{2}}^-)^T A^- U_{\frac{1}{2}}^- - 2\delta (U_{\frac{1}{2}}^-)^T A^- U_{\frac{1}{2}}^+. 
\end{align*}
Clearly, there exists a constant $c_0>0$ satisfying
\begin{align*}
-(U_{\frac{1}{2}}^-)^T A^+ U_{\frac{1}{2}}^- =& -(R_+^TU_{\frac{1}{2}}^-)^T \Lambda_+ (R_+^TU_{\frac{1}{2}}^-) \leq -c_0 |R_+^TU_{\frac{1}{2}}^-|^2, \\[3mm]
\delta (U_{\frac{1}{2}}^-)^T A^- U_{\frac{1}{2}}^- =& ~\delta(R_-^TU_{\frac{1}{2}}^-)^T \Lambda_- (R_+^TU_{\frac{1}{2}}^-) \leq -c_0 \delta |R_-^TU_{\frac{1}{2}}^-|^2.
\end{align*}  
Besides, there exists a constant $c_1>0$ such that
$$
-2\delta (U_{\frac{1}{2}}^-)^T A^- U_{\frac{1}{2}}^+ = -2 \delta (R_-^TU_{\frac{1}{2}}^-)^T \Lambda_- (R_-^TU_{\frac{1}{2}}^+) \leq \frac{c_0\delta}{2}|R_-^TU_{\frac{1}{2}}^-|^2 + c_1 \delta |R_-^TU_{\frac{1}{2}}^+|^2.
$$
By the coupling condition in \eqref{DGBC}, there exists a constant $c_2>0$ such that 
$$
|R_-^TU_{\frac{1}{2}}^+|^2 \leq c_2 (|R_+^TU_{\frac{1}{2}}^-|^2 + |P_-^Tu_{\frac{1}{2}}^+|^2).
$$
Hence the boundary term on the right of \eqref{eq:theorem3.1-step1} satisfies
\begin{align}\label{est1} 
2J_1(U_{\frac{1}{2}}^-,U_{\frac{1}{2}}^+)\leq -(c_0-c_1c_2 \delta) |R_+^TU_{\frac{1}{2}}^-|^2 - \frac{c_0 \delta}{2} |R_-^TU_{\frac{1}{2}}^-|^2 + c_1c_2 \delta |P_-^Tu_{\frac{1}{2}}^+|^2.
\end{align}

Now we turn to estimate the boundary term in \eqref{eq:theorem3.1-step2}. By calculation, 
\begin{align*}
-2 j_2(u^-_{\frac{1}{2}},u^+_{\frac{1}{2}}) =& - \big[ \delta (u_{\frac{1}{2}}^+)^T A_{11}^+ u_{\frac{1}{2}}^+ + (u_{\frac{1}{2}}^+)^T A_{11}^- u_{\frac{1}{2}}^+\big] + 2(u_{\frac{1}{2}}^+)^T\big[\delta A_{11}^+ u_{\frac{1}{2}}^- + A_{11}^- u_{\frac{1}{2}}^+ \big]\\[3mm]
=& -\delta (u_{\frac{1}{2}}^+)^T A_{11}^+ u_{\frac{1}{2}}^+ + (u_{\frac{1}{2}}^+)^T A_{11}^- u_{\frac{1}{2}}^+ + 2 \delta (u_{\frac{1}{2}}^+)^T A_{11}^+ u_{\frac{1}{2}}^-.
\end{align*}
Similarly, we can find constants $\tilde{c}_0$, $\tilde{c}_1>0$ such that 
\begin{align*}
-2 j_2(u^-_{\frac{1}{2}},u^+_{\frac{1}{2}}) \leq &-\frac{\tilde{c}_0}{2}\delta|P_+^Tu_{\frac{1}{2}}^+|^2 - \tilde{c}_0|P_-^Tu_{\frac{1}{2}}^+|^2 + \tilde{c}_1 \delta |P_+^Tu_{\frac{1}{2}}^-|^2.
\end{align*}
According to the coupling condition \eqref{DGBC}, the last term on the right hand side of the above inequality is bounded by 
$$
|P_+^Tu_{\frac{1}{2}}^-|^2 \leq \tilde{c}_2 (|P_-^Tu_{\frac{1}{2}}^+|^2+|R_+^TU_{\frac{1}{2}}^-|^2)
$$
with $\tilde{c}_2>0$ a constant.
Thus, the boundary term on the right hand side of \eqref{eq:theorem3.1-step2} satisfies
\begin{align*}
-2 j_2(u^-_{\frac{1}{2}},u^+_{\frac{1}{2}}) \leq &-\frac{\tilde{c}_0}{2}\delta|P_+^Tu_{\frac{1}{2}}^+|^2 - (\tilde{c}_0-\tilde{c}_1\tilde{c}_2 \delta)|P_-^Tu_{\frac{1}{2}}^+|^2 + \tilde{c}_1\tilde{c}_2 \delta |R_+^TU_{\frac{1}{2}}^-|^2.
\end{align*}
Combining this estimate with \eqref{est1}, we know that for sufficiently small $\delta>0$ 
$$
2J_1(U_{\frac{1}{2}}^-,U_{\frac{1}{2}}^+) - 2j_1(u_{\frac{1}{2}}^-,u_{\frac{1}{2}}^+) \leq 0.
$$
It is easy to see from the above proof that  $\delta$ only depends on the coefficient matrices in the problem.
Therefore, it follows from \eqref{eq:theorem3.1-step1} and \eqref{eq:theorem3.1-step2} that 
$$
\frac{d}{dt}\left(\|U_h\|^2_{\delta_-} + \|u_h\|^2_{\delta_+}\right)\leq 2C_\delta \|U_h\|^2_{\delta_-}.
$$
At last, by Gronwall's inequality, the estimate stated in the theorem holds. 
\end{proof}

\begin{remark}
From the weighted $L^2$ stability in Theorem \ref{thm2.3},
    it is easy to show the $L^2$ stability under the same assumption:
$$
\|U_h(t)\|^2_{L^2(\mathbb{R}^-)}+\|u_h(t)\|^2_{L^2(\mathbb{R}^+)} \leq C(T) \left(\|U_h(0)\|^2_{L^2(\mathbb{R}^-)} + \|u_h(0)\|^2_{L^2(\mathbb{R}^+)}\right),
$$
for any $t\in[0,T]$ with $C=C(T)>0$.
\end{remark}

\section{Applications and numerical examples}\label{sec:numerical-example}

In this section, we explicitly derive the coupling conditions for the linearized Carleman model and the Grad's moment system. We validate our analysis and the effectiveness of the DG scheme using several numerical examples.

\subsection{Carleman model}
The interface problem for Carleman model \cite{carleman1957problemes} reads as 
\begin{align*}
&\partial_t f_+ + v \partial_x f_+ = \frac{1}{\epsilon(x)}(f_-^2-f_+^2)\\[2mm]
&\partial_t f_- - v \partial_x f_- = \frac{1}{\epsilon(x)}(f_+^2-f_-^2).
\end{align*}
Here $v>0$ is a constant, $\epsilon(x)=1$ for $x<0$ and $\epsilon(x)=\epsilon>0$ is a small constant for $x>0$.	 
Let $\rho=f_+ + f_-$ and $q= f_--f_+$, we can rewrite the above system as
$$
\partial_t\begin{pmatrix}
\rho \\[1mm]
q
\end{pmatrix} + \begin{pmatrix}
0 & -v \\[1mm]
-v & 0
\end{pmatrix} \partial_x\begin{pmatrix}
\rho \\[1mm]
q
\end{pmatrix} = \begin{pmatrix}
0 & 0  \\[1mm]
0 & -2\rho
\end{pmatrix} \begin{pmatrix}
\rho \\[1mm]
q
\end{pmatrix}/\epsilon(x).
$$
Considering its linearized version at $\rho=\rho_*$ with $\rho_*>0$ a constant, we obtain 
\begin{equation}\label{original-Carlman}
	U_t + AU_x = QU/\epsilon(x)
\end{equation}
with 
$$
U = \begin{pmatrix}
\rho \\[1mm]
q
\end{pmatrix}, \qquad 
A = \begin{pmatrix}
0 & -v \\[1mm]
-v & 0
\end{pmatrix}, \qquad 
Q = \begin{pmatrix}
0 & 0 \\[1mm]
0 & -2\rho_*
\end{pmatrix}.
$$

In our numerical test, we take $v=1$ and $\rho_*=1/2$. In the domain decomposition (DD) method, we solve 
\begin{equation}\label{eq:carleman-left}
    U^l_t + AU^l_x = QU^l     
\end{equation}
on the left domain 
and the corresponding equilibrium system 
\begin{equation}\label{eq:carleman-right}
    \partial_t\rho^r = 0
\end{equation}
on the right domain. To explicitly derive the coupling condition, we compute each term in \eqref{coco} and the result is listed in Table \ref{table:1}.
\begin{table}[!hbp]
\centering
\begin{tabular}{c|c|c|c|c|c|c|c|c}
\hline
&&&&&&&\\[-1em]
$R_+$ & $R_-$ & $P_+$ & $P_-$ & $P_0$ & $R_S$ & $K$ & $\tilde{K}$ & N\\ 
\hline
&&&&&&&\\[-1em]
$(1,-1)^T$ & $(1,1)^T$ & --- & --- & $1$ & --- & $-1$ & --- & ---\\ 
\hline
\end{tabular}
\caption{Carleman model: matrices in the coupling condition \eqref{coco}.}
\label{table:1}
\end{table}
Substituting the matrices in Table \ref{table:1} into the coupling condition \eqref{coco}, we arrive at
$$
\begin{pmatrix}
-1 & 1 \\[2mm]
-1 & 0
\end{pmatrix}
\begin{pmatrix}
(\rho^l+q^l)/2 \\[2mm]
\hat{w}
\end{pmatrix} = \frac{1}{2}\begin{pmatrix}
1 \\[2mm]
-1
\end{pmatrix} (\rho^l-q^l) - \begin{pmatrix}
\rho^r \\[2mm]
0
\end{pmatrix}.
$$
From the above equation, we can solve out
\begin{equation}\label{eq:carleman-coupling}
  q^l = 0  
\end{equation}
Clearly, no boundary condition is needed for \eqref{eq:carleman-right} on the right domain $x>0$ and the boundary condition $q^l=0$ is given for \eqref{eq:carleman-left} on the left domain $x<0$.

We use this simple example to test the convergence rate with respect to $\epsilon$ which is proved in Theorem \ref{thm2.3}. To this end, we will compute the solution $(\rho^{\epsilon}, q^{\epsilon})$ to the original problem \eqref{original-Carlman} and the solution $(\rho^D, q^D)$ to the coupling problem \eqref{eq:carleman-left}-\eqref{eq:carleman-right}-\eqref{eq:carleman-coupling}. We take the initial data as
$$
\rho(x,0)=\sin(x)+1,\quad q(x,0)=0.
$$
In this example, we use the first-order upwind scheme in space and the forward Euler in time to solve both \eqref{original-Carlman} and \eqref{eq:carleman-left}-\eqref{eq:carleman-coupling}.
The computational domain is $[-0.4,0.4]$. To reduce the numerical error, we take the mesh size small enough $\Delta x=4\times 10^{-6}$. The solution is computed up to $t=0.2$. Since we only care about the interface phenomena, we will neglect the effects of the boundaries at $x=\pm 0.4$ and only focus on the solutions in the inner domain $[-0.1, 0.1]$.

The numerical solutions for $-0.1\le x \le 0.1$ are plotted in Figure \ref{fig1}. The dashed line represents for the solution of \eqref{eq:carleman-left}-\eqref{eq:carleman-coupling} by the domain decomposition (DD) method and the solid lines represent for the solutions to the original problem \eqref{original-Carlman} with different values of $\epsilon$. The figure shows that, as $\epsilon$ goes to zero, the solution of \eqref{original-Carlman} converges to the solution of the DD method. Moreover, from this figure, we see that there only exist $O(\sqrt{\epsilon})$ boundary-layers for the first variable $\rho$. In fact, recall Proposition \ref{prop2.1} that: the $O(\sqrt{\epsilon})$ boundary layers are in the formulation of $
\begin{pmatrix}
P_0\\
0
\end{pmatrix}
\hat{w}$ and the $O(\epsilon)$ boundary layers are $\begin{pmatrix}
N\\
\tilde{K}
\end{pmatrix}\tilde{w}$. By the calculations of $P_0$, $N$ and $\tilde{K}$ in Table \ref{table:1}, we know that there is no $O(\epsilon)$ boundary-layer and there exists $O(\sqrt{\epsilon})$ boundary-layer only for the first variable. 

\begin{figure}
\begin{minipage}{0.5\linewidth}
\centering
\includegraphics[width=3.0in]{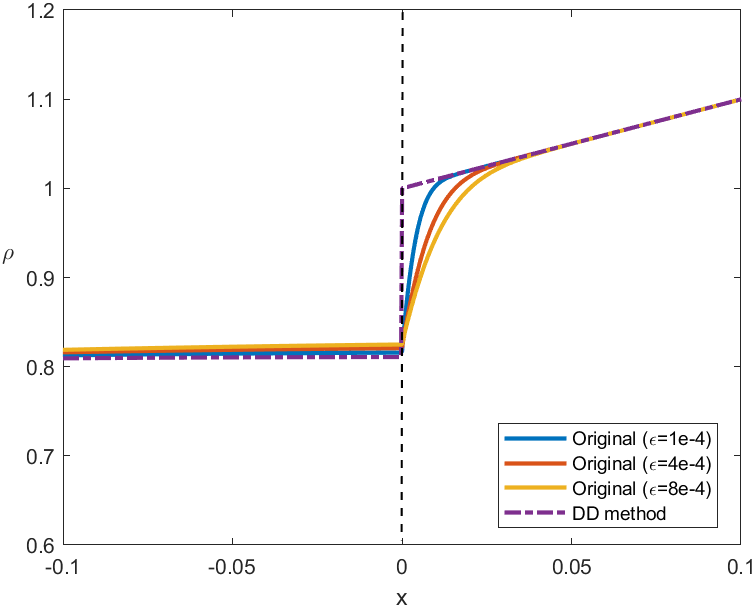}
\end{minipage}
\begin{minipage}{0.5\linewidth}
\centering
\includegraphics[width=3.0in]{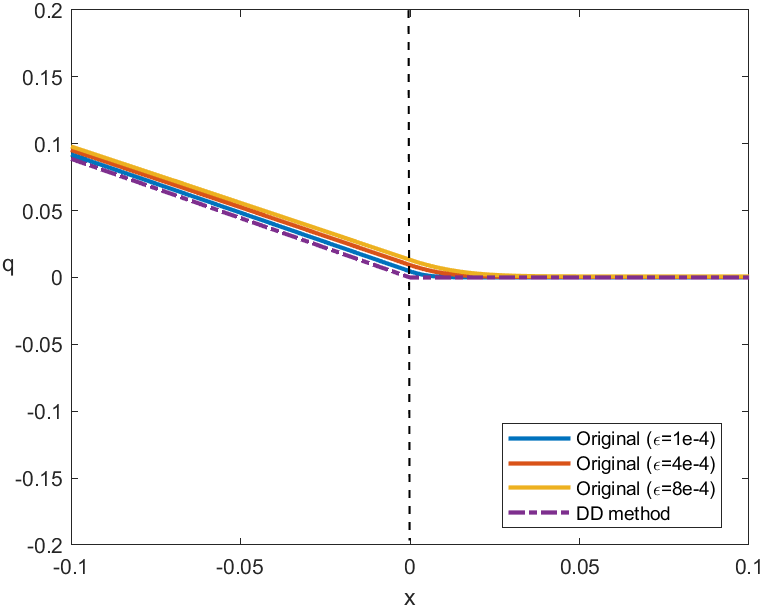}
\end{minipage}
\caption{Carleman model: numerical solutions of $\rho$ and $q$ for the DD method and the original problem with different values of $\epsilon$.} 
\label{fig1}
\end{figure}

In Table \ref{table:convergence-epsilon}, we show the $L^2$ errors between the solution of the original problem $(\rho^\epsilon,q^\epsilon)$ and the solution of the DD method $(\rho^D,q^D)$ on the domain $[-0.1,0.1]$ with different values of $\epsilon$. From Table \ref{table:convergence-epsilon}, we notice that $\|\rho^\epsilon-\rho_D\|_{L^2}=O(\epsilon^{1/4})$ and $\|q^\epsilon-q_D\|_{L^2}=O(\epsilon^{1/2})$, which verify our analysis in Theorem \ref{thm2.3} and Remark \ref{rmk25}.
\begin{table}[H]
\centering
\begin{tabular}{c|c|c|c|c}
\hline
&&&&\\[-1em]
$\epsilon$&8e-4&4e-4&2e-4&1e-4\\
\hline
&&&&\\[-1em]
$\|\rho^\epsilon-\rho^D\|_{L^2}$&1.32e-2&1.12e-2&9.52e-3&8.06e-3\\
\hline
&&&&\\[-1em]
convergence order&--- & 0.24 & 0.24  &0.24\\
\hline
&&&&\\[-1em]
$\|q^\epsilon-q^D\|_{L^2}$&4.30e-2&3.01e-2&2.11e-2&1.47e-2\\
\hline
&&&&\\[-1em]
convergence order&---& 0.52 & 0.51 & 0.51\\
\hline
\end{tabular}
\caption{Carleman model: errors between $(\rho^\epsilon,q^\epsilon)$ and $(\rho^D,q^D)$ and convergence rates with respect to $\epsilon$.}
\label{table:convergence-epsilon}
\end{table}

\subsection{Grad's moment system}
We consider the linearized Grad's  moment system in 1D \cite{CFL,ZYL,Grad}:
\begin{equation}\label{moment-eq}
	U_t + A U_x = QU/\epsilon(x)
\end{equation}
with 
$$
U=
\begin{pmatrix}
\rho\\[2mm]
w\\[2mm]
\theta/\sqrt{2}\\[2mm]
\sqrt{3!}f_3\\ 
\vdots\\[2mm]
\sqrt{M!}f_M
\end{pmatrix},\quad
A = 
\begin{pmatrix}
~0 & ~1 & ~~ & ~~ & ~~ & ~~~\\[2mm]
~1 & ~0 & \sqrt{2} & ~~ & ~~ & ~~~\\[2mm]
~~ & \sqrt{2} & ~0 & \sqrt{3}~ & ~~ & ~~~\\ 
~~ & ~~ & \sqrt{3} & ~0~ & \ddots & ~~~\\ 
~~ & ~~ & ~~ & \ddots & ~0 & \sqrt{M} \\[2mm]
~~ & ~~ & ~~ & ~ & \sqrt{M} & 0 
\end{pmatrix},\quad Q= -\text{diag}(0,0,0,\underbrace{1,...,1}_{M-2}).
$$
In the above equation, $\rho$ is the density, $w$ is the macro velocity, $\theta$ is the temperature and $f_3,...,f_M$ with $M\ge 3$ are high order moments. The moment system is obtained by taking moments on the both sides of the Bhatnagar-Gross-Krook (BGK) model. Here we only consider its linearized one-dimensional version. For the sake of invertibility of $A$ in \eqref{moment-eq}, we always assume $M$ to be odd.
As to the interface problem, the parameter $\epsilon(x)=1$ for $x<0$ and $\epsilon(x)=\epsilon\ll 1$ for $x>0$. 

In the DD method, we solve 
\begin{equation}\label{eq:moment-left}
    \partial_tU^l + A\partial_xU^l = QU^l   
\end{equation}
with 
$$
U^l=(\rho^l,w^l,\theta^l/\sqrt{2},\sqrt{3!}f^l_3,...,\sqrt{M!}f^l_M)^T
$$ 
for $x<0$ and solve the corresponding equilibrium system 
\begin{equation}\label{eq:moment-right}
\partial_t u^r + 
A_{11}
\partial_x
u^r = 0  
\end{equation}
with 
$$
u^r=
\begin{pmatrix}
\rho^r \\[1mm]
w^r \\[1mm]
\theta^r/\sqrt{2}
\end{pmatrix}
,\qquad
A_{11}=
\begin{pmatrix}
0 & 1 & 0 \\[1mm]
1 & 0 & \sqrt{2} \\[1mm]
0 & \sqrt{2} & 0
\end{pmatrix}
$$
for $x>0$. To derive the coupling conditions between \eqref{eq:moment-left} and \eqref{eq:moment-right}, we compute each term in \eqref{coco} as 
$$
P_0=\begin{pmatrix}
\frac{\sqrt{6}}{3} & ~0 & -\frac{\sqrt{3}}{3}
\end{pmatrix}^T,\quad
P_- = \begin{pmatrix}-\frac{\sqrt{6}}{6} & \frac{\sqrt{2}}{2} & -\frac{\sqrt{3}}{3}\end{pmatrix}^T,\quad 
P_+ = \begin{pmatrix}\frac{\sqrt{6}}{6} & \frac{\sqrt{2}}{2} & \frac{\sqrt{3}}{3}\end{pmatrix}^T
$$
and
$$
K = A_{12}^TP_0 = (-1,\underbrace{0,...,0}_{M-3})^T,\qquad \tilde{K} = (0,I_{M-3})^T,\qquad N = \begin{pmatrix}
 \frac{2\sqrt{6}}{3} & 0 & \cdots & 0\\[1mm] 
 0 & 0 & \cdots & 0\\[1mm]
 -\frac{2\sqrt{3}}{3} & 0 & \cdots & 0
\end{pmatrix}.
$$
Moreover, in order to obtain $R_S$, we compute
$$
\tilde{K}^TX\tilde{K} = 
\begin{pmatrix}
 ~0 & \sqrt{5}~ & ~~ & ~~~\\ 
 \sqrt{5} & ~0~ & \ddots & ~~~\\ 
~~ & \ddots & ~0 & \sqrt{M} \\[2mm]
~~ & ~ & \sqrt{M} & 0 
\end{pmatrix},\qquad \tilde{K}^TS\tilde{K} = -I_{M-3}.
$$
By our construction, $R_S$ is the right-stable matrix for $(\tilde{K}^TX\tilde{K})^{-1}\tilde{K}^TS\tilde{K}=-(\tilde{K}^TX\tilde{K})^{-1}$. Thus we can take $R_S$ as the eigenvectors of $\tilde{K}^TX\tilde{K}$ associated with its positive eigenvalues. {From the above expression, one can check that the matrix $\tilde{K}^TX\tilde{K}$ has $\frac{M-3}{2}$ positive eigenvalues and thereby the matrix $R_S$ should be of size
$(M-3)\times\frac{M-3}{2}$.}

At last, we compute $R_+$ and $R_-$ which are eigenvectors of $A$ associated with its positive eigenvalues and negative eigenvalues. Note that the characteristic polynomial of $A$ is $(M+1)$-order Hermite polynomial $He_{M+1}(x)$. For each fixed $M$, we can explicitly express $R_+$ and $R_-$ in terms of the zeros of $He_{M+1}(x)$. We refer readers to \cite{CFL} for more details. At this point, the coefficient matrices in \eqref{coco} are explicitly computed and thus the coupling condition \eqref{eq:coupling-condition-B} is obtained.

Next we try to explicitly express the coupling condition in terms of the moment variables $(\rho^l,w^l,\theta^l,f_3^l,\dots,f_M^l)$ and $(\rho^r,w^r,\theta^r)$. 
Denote 
$$
\Phi=\begin{pmatrix}
P_-^T & ~0~ & ~0~\\[1mm]
0 & 1 & ~0~ \\[1mm]
0 & 0 & R_U^T
\end{pmatrix}
$$
with $R_U$ satisfying $R_U^TR_S=0$. Then it is easy to see that 
$$
\Phi\begin{pmatrix}
P_+ & P_0 & NR_S\\[2mm]
0 & 0 & \tilde{K}R_S
\end{pmatrix}
=
\Phi\begin{pmatrix}
P_+ & P_0 & NR_S\\[1mm]
0 & 0 & 0 \\[1mm]
0 & 0 & R_S
\end{pmatrix} = 0.
$$
Multiplying $\Phi$ on both sides of \eqref{coco} yields the boundary condition for $U^l$:
\begin{equation*}
\Phi U^l(0,t) = \begin{pmatrix}
\beta_- \\[1mm]
0
\end{pmatrix}.
\end{equation*}
Or equivalently, 
\begin{eqnarray}\label{moment:coup1}
\left\{\begin{array}{l}
\rho^l - \sqrt{3} w^l + \theta^l = \rho^r - \sqrt{3} w^r + \theta^r, \\[3mm]
f_3^l = 0,\\[3mm]
(\sqrt{4!}f^l_4,...,\sqrt{M!}f^l_M)R_U=0.
\end{array}\right.
\end{eqnarray}
Next, we multiply $(P_+^T, \underbrace{0,...,0}_{M-2})$ on both sides of \eqref{coco} to obtain:
$$
\beta_+ = P_+^T u^l(0,t).
$$
That is, 
\begin{equation}\label{moment:coup2}
	\rho^r + \sqrt{3} w^r + \theta^r = \rho^l + \sqrt{3} w^l + \theta^l.
\end{equation}
The relations \eqref{moment:coup1} and \eqref{moment:coup2} are coupling conditions for \eqref{eq:moment-left} and \eqref{eq:moment-right}. 
For $M=5$, the right-stable matrix $R_S=(1,1)^T$ and thereby $R_U = (1,-1)^T$. In this case, the coupling condition reads as 
\begin{eqnarray}\label{Moment-coup}
\left\{\begin{array}{l}
\rho^l - \sqrt{3} w^l + \theta^l = \rho^r - \sqrt{3} w^r + \theta^r, \\[3mm]
f_3^l = 0,\\[3mm]
f_4^l - \sqrt{5}f^l_5=0, \\[3mm]
\rho^r + \sqrt{3} w^r + \theta^r = \rho^l + \sqrt{3} w^l + \theta^l.
\end{array}\right.
\end{eqnarray}

Next we solve the interface problem \eqref{eq:moment-left}-\eqref{eq:moment-right}-\eqref{Moment-coup} for the moment closure system with $M=5$. We use the DG scheme presented in the previous section for the spatial discretization and the third-order Runge-Kutta (RK) method \cite{shu1988efficient} for the time discretization.  
The initial data are given by
\begin{align*}
(\rho, w, \theta)(x,0) = (\sin(2x)+1.1, ~0, ~\sqrt{2}),\qquad (f_3,f_4,f_5)(x,0)=(0,0,0).
\end{align*}
We take the polynomial degree in the DG scheme to be $2$. The computational domain is $[-2\pi,2\pi]$. Similar to the treatment in the Carleman model, we avoid the boundary effects at $x=\pm 2\pi$ by truncating to a smaller domain. In the numerical tests, we take the CFL number to be $0.17$ for \eqref{eq:moment-left}.

We firstly test the convergence rates of the RKDG scheme solving DD problem by comparing the numerical solutions at $t=0.5$ with the reference solution (solved by RKDG itself with a refined mesh $\Delta x=\pi/640$). In Table \ref{table3}, $U^l$ and $u^r$ are the numerical solutions of \eqref{eq:moment-left} and \eqref{eq:moment-right}, $U^l_e$ and $u^r_e$ are reference solutions of \eqref{eq:moment-left} and \eqref{eq:moment-right}. The $L^2$-norm is computed over the domain $[-2\pi/3,2\pi/3]$. From Table \ref{table3}, we can see that the RKDG scheme achieves the third-order accuracy.
\begin{table}[H]
\centering
\begin{tabular}{c|c|c|c|c|c}
\hline
&&&&\\[-1em]
$\Delta x$ & $\pi/10$ & $\pi/20$ & $\pi/40$ & $\pi/80$ & $\pi/160$\\ 
\hline
&&&&\\[-1em]
$\|U^l-U^l_e\|_{L^2}$ & 5.04e-3 & 6.78e-4 & 8.52e-05 & 1.08e-05 & 1.36e-06\\
\hline
&&&&\\[-1em]
Order & --- & 2.89 & 2.99 & 2.98 & 2.98\\
\hline
&&&&\\[-1em]
$\|u^r-u^r_e\|_{L^2}$ & 9.77e-4 & 1.23e-4 &	1.54e-05 & 1.93e-06	& 2.41e-07\\
\hline
&&&&\\[-1em]
Order & --- & 2.99 & 3.00 & 3.00 & 3.00\\
\hline
\end{tabular}
\caption{Grad’s moment system: $L^2$ errors and convergence rates of the RKDG scheme.}
\label{table3}
\end{table}

\begin{figure}
\begin{minipage}{0.5\linewidth}
\centering
\includegraphics[width=3.0in]{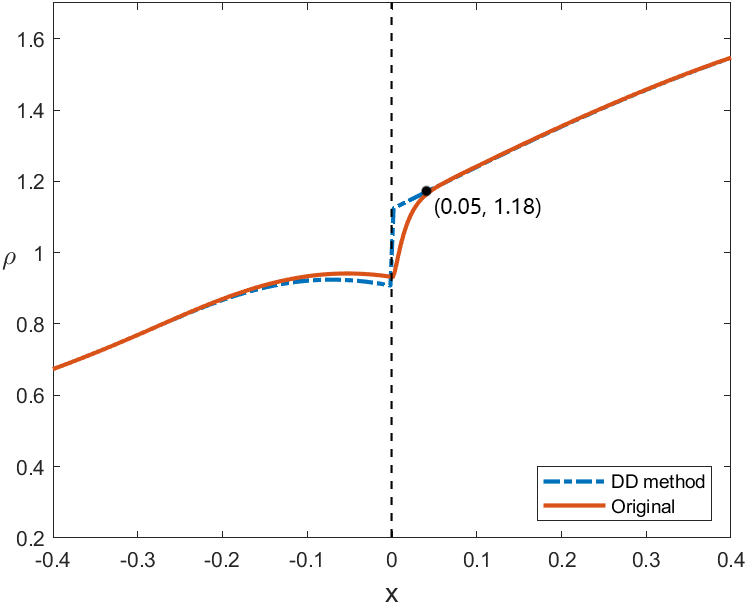}
\end{minipage}%
\begin{minipage}{0.5\linewidth}
\centering
\includegraphics[width=3.0in]{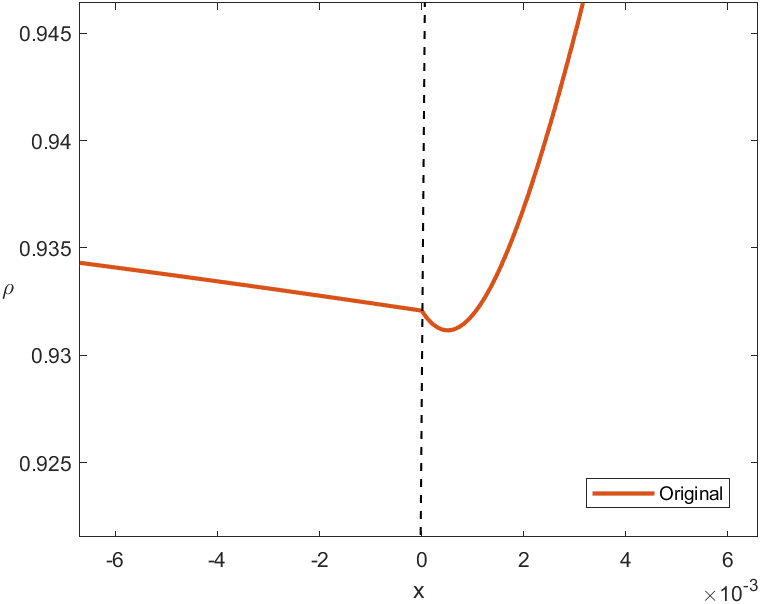}
\end{minipage}
\begin{minipage}{0.5\linewidth}
\centering
\includegraphics[width=3.0in]{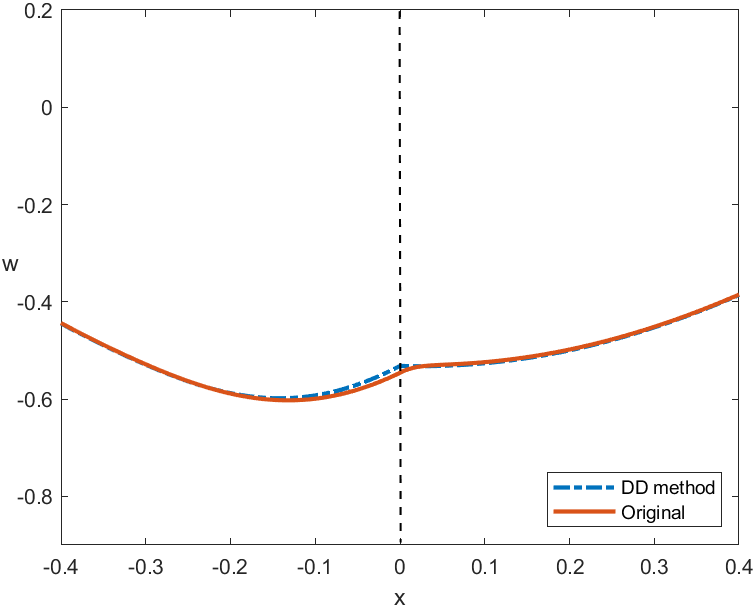}
\end{minipage}%
\begin{minipage}{0.5\linewidth}
\centering
\includegraphics[width=3.0in]{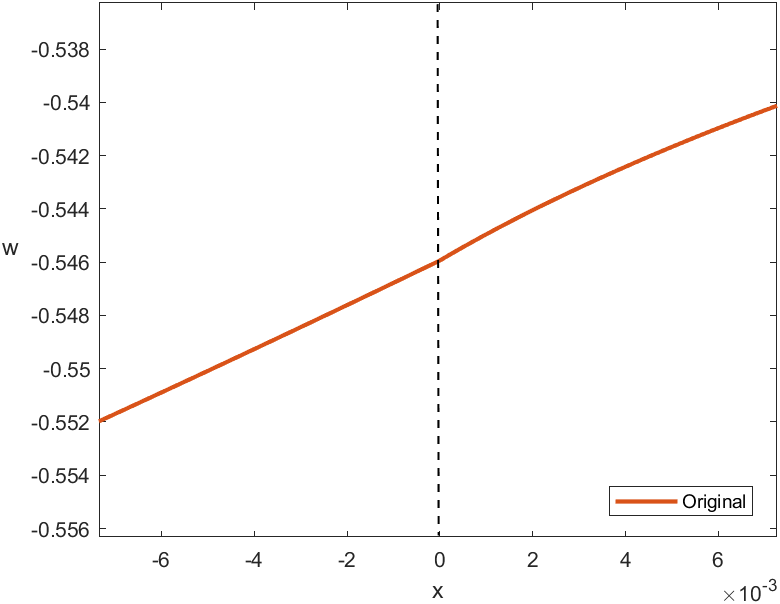}
\end{minipage}
\begin{minipage}{0.5\linewidth}
\centering
\includegraphics[width=3.0in]{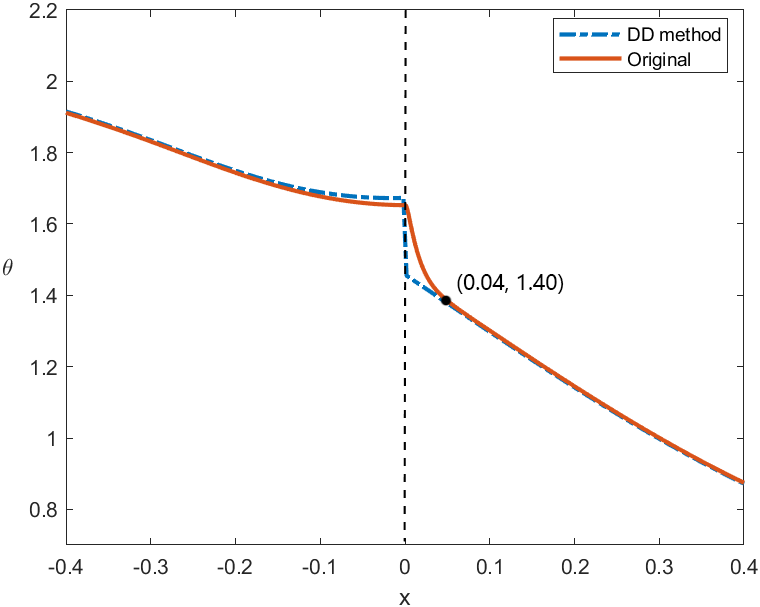}
\end{minipage}%
\begin{minipage}{0.5\linewidth}
\centering
\includegraphics[width=3.0in]{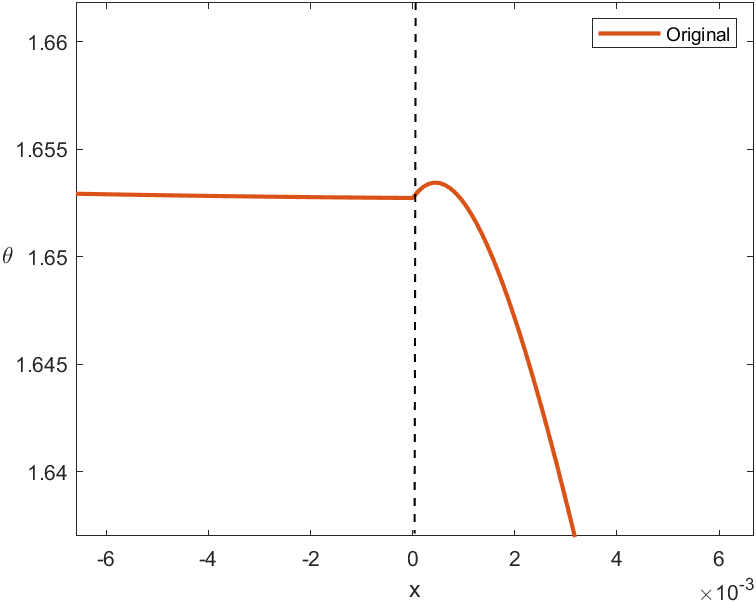}
\end{minipage}
\caption{Grad’s moment system: numerical solution of $(\rho,w,\theta)$. The first, second and third row represent for $\rho$, $w$ and $\theta$ respectively. The figures on the right are enlarged to capture the boundary-layers in the scale $O(\epsilon)$}
\label{fig2}
\end{figure}

Furthermore, we plot the numerical solutions of the original problem \eqref{moment-eq} and the solutions obtained by the domain decomposition (DD) method \eqref{eq:moment-left}-\eqref{eq:moment-right}-\eqref{Moment-coup}. In the original problem, the parameter $\epsilon$ is taken as $10^{-3}$. The equation \eqref{moment-eq} is solved by the first-order upwind scheme in space and the forward Euler method in time with sufficiently small mesh size $\Delta x = 10^{-4}$ and the CFL number to be $0.67$. 
The coupling problem \eqref{eq:moment-left}-\eqref{eq:moment-right}-\eqref{Moment-coup} is solved by the DG method with $\Delta x=\pi/80$. The results are plotted in Figure \ref{fig2}. The red solid line represents for the solution of the original problem while the blue dashed line represents for the solution of the DD method. 
The solutions in the computational domain are presented in the left part, while the right part are the zoom of the solutions near the interface $x=0$.

From the left figures, we see that the numerical solutions obtained from the DD method approximate well
to the original problem \eqref{moment-eq} for $\epsilon=10^{-3}$. In particular, the variables $\rho$ and $\theta$ allow $O(\sqrt{\epsilon})$ boundary-layers. 
From the right figures, we see that there also exist $O(\epsilon)$ boundary-layers for $\rho$ and $\theta$. Note that the limit for the $x$-axis is in the order of $10^{-3}$. Recall the theoretical result in Proposition \ref{prop2.1}: the $O(\sqrt{\epsilon})$ boundary layers are in the formulation of $P_0\hat{w}$ and the $O(\epsilon)$ boundary layers are $N\tilde{w}$. From the calculations of $P_0$ and $N$, we see that  there is indeed no boundary-layer for the second variable but exist boundary-layers for the first and the third variables.




\newpage
\begin{appendices}
\section{Review of the Generalized Kreiss Condition}\label{appedix:GKC}
In this appendix, we briefly review the Generalized Kreiss Condition (GKC) for initial-boundary value problems (IBVPs) for hyperbolic relaxation systems \eqref{problem} which was introduced in \cite{Y2}. More details can be found in \cite{Y2,ZY}.

The motivation of the GKC is to guarantee the existence of zero-relaxation limit for the IBVPs for hyperbolic relaxation systems. It was observed in \cite{Y2} that the structural stability condition \cite{Y1} for the Cauchy problem is not sufficient to guarantee the existence.
Consider the following linear IBVP with constant coefficients
\begin{eqnarray}\label{Append-IBVP}
\left\{ 
\begin{array}{l}
\partial_tU+A\partial_xU=QU/\epsilon+\bar{Q}U,\qquad x>0,\\[3mm]
BU(0,t)=b(t),\\[3mm]
U(x,0)=U_0(x).
\end{array}\right.	
\end{eqnarray}
The coefficient matrices $A$, $Q$ are assumed to satisfy the structural stability condition \cite{Y1}. Without loss of generality, we may assume that $A$ and $Q$ are both symmetric and $Q=\textrm{diag}(0,S)$ with $S<0$. Corresponding to the partition of $Q$, we write
$$
A=
\begin{pmatrix}
A_{11} & A_{12} \\[1mm]
A_{12}^T & A_{22}
\end{pmatrix}
$$
Here we only discuss the case with non-characteristic boundary, i.e., $A$ is invertible in \eqref{Append-IBVP}. Note that the boundary may be characteristic for the reduced system. Namely, there could exist zero eigenvalues for $A_{11}$.
The GKC for \eqref{Append-IBVP} reads as
\begin{definition}[{Generalized Kreiss condition} \cite{Y2}] The boundary matrix $B$ in \eqref{Append-IBVP} satisfies the GKC, if there exist a constant $c_K>0$, such that
$$
|\det\{BR_M^S(\xi,\eta)\}|\geq c_K,
$$
for all $\eta\geq 0$ and   $\xi$ with $\textrm{Re}(\xi)>0$.
Here $R_M^S(\xi,\eta)$ is the right-stable matrix for $M(\xi,\eta):=A^{-1}(\eta Q-\xi I)$.
\end{definition}
\noindent Note that the right-stable matrix is defined by
\begin{definition}[Right-stable/unstable matrix]
Let the matrix $A\in\mathbb{R}^{n\times n}$ have precisely $k\ (0\leq k\leq n)$ stable
eigenvalues (i.e. eigenvalues with negative real parts). The full-rank matrix $R^S \in \mathbb{R}^{n\times k}$
is called a right-stable matrix of $A$ if
$$AR^S=R^SS$$
where $S$ is a $k\times k$ stable matrix. Similarly we introduce the right-unstable matrix $R^U$.
\end{definition}

Having the GKC, a natural question is how to check this condition. It was proved in \cite{Y2} that the following strictly dissipative condition implies the GKC:
\begin{definition}[{Strictly dissipative condition}]
The boundary matrix $B$ in \eqref{Append-IBVP} satisfies the strictly dissipative condition, if there is a constant $c>0$ such that 
$$
y^TAy\leq -c|y|^2+c^{-1}|By|^2 
$$
for all $y \in \mathbb{R}^n$.
\end{definition} 

In fact, in the proof of Theorem \ref{coth}, we need to use the following proposition.
\begin{prop}\label{prop:appendix-results}
The following results hold:
\begin{itemize}
\item[(i)] The eigenvector matrix $R_+^T$ of $A$ defined in \eqref{eq:R+R-} satisfies the GKC.
\item[(ii)] Taking $\eta\rightarrow \infty$ and $\xi=1$ in $R_M^S(\xi,\eta)$ in the GKC, we have the expression
$$
R_M^S(1,\infty)=
\begin{pmatrix}
P_+ & P_0 & NR_S\\[2mm]
0 & 0 & \tilde{K}R_S
\end{pmatrix},
$$
where the matrices $P_+$, $P_0$, $N$, $\tilde{K}$ and $R_S$ are defined in Section \ref{section2.2}.
\end{itemize}
\end{prop}
\begin{proof}
\begin{itemize}
\item[(i)]
Recall the definition of $R=(R_+,R_-)$. 
It is easy to see that
\begin{equation*}
  R_+R_+^T - c R_+ \Lambda_+ R_+^T - c R_- \Lambda_- R_-^T - c^2 I\\[2mm]
= 
R \begin{pmatrix}
I - c \Lambda^+ & \\[1mm]
                & - c \Lambda^-
\end{pmatrix}
R^T - c^2 I
\end{equation*}
is positive definite when $c$ is sufficiently small. 
Thus $R_+^T$ indeed satisfies the strictly dissipative condition and thereby satisfies the GKC, i.e.,  $|\det\{BR_M^S(\xi,\eta)\}|\geq c_K$ for any $\xi$ with $Re\xi>0$ and $\eta\geq 0$ with some $c_K>0$.
\item[(ii)]
The proof to the second result is obtained by exploiting the analysis of $R_M^S(\xi,\eta)$ in \cite{ZY} and the interested readers may refer to this paper for more details.
\end{itemize}
\end{proof}

Moreover, the proof of Theorem \ref{thm2.3} uses the following proposition:
\begin{prop}\label{prop:appendix-estimate}
Assume that $A$, $Q$, $\bar{Q}$ in \eqref{Append-IBVP} are symmetric, $Q=\text{diag}(0,S)$ with $S<0$ and $A$ is invertible. Moreover, assume that the GKC holds and $U_0(x)\equiv 0$.
For sufficiently small $\epsilon<\epsilon_0\ll 1$, the solution to \eqref{Append-IBVP} satisfies the following estimate 
$$
\|U(\cdot,t)\|^2_{L^2(\mathbb{R}^+)} + \int_0^T|U(0,t)|^2dt \leq C \int_0^T|b(t)|^2dt .
$$
Here, the constant $C>0$ is independent of $\epsilon$.
\end{prop}
\begin{proof}
We adapt the method in \cite{GKO} to obtain the estimate. Define the Laplace transform with respect to $t$:
$$
\hat{U}(x,\xi)=\int_0^{\infty}e^{-\xi t}U(x,t)dt,\qquad Re \xi>0. 
$$
Then we deduce that 
\begin{align}\label{5.16}
\left\{
{\begin{array}{*{20}l}
\vspace{2mm}\partial_{x}\hat{U} =A^{-1}\left(\eta Q + \bar{Q} - \xi I_n \right)\hat{U},\qquad \eta=1/\epsilon,\\[2mm]
\vspace{2mm} B\hat{U}(0,\xi)=\hat{b}(\xi),\\[2mm]
\|\hat{U}(\cdot,\xi)\|_{L^2(\mathbb{R}^+)}<\infty \quad \text{for ~a.e.} ~~\xi. 
\end{array}}
\right.
\end{align}
Let $s=\sqrt{\eta^2+|\xi|^2}$ and $\xi'=\xi/s$, $\eta'=\eta/s$, we can denote 
$$
\bar{M}(\xi',\eta',1/s)=A^{-1}\left(\eta' Q + \frac{1}{s}\bar{Q} - \xi' I_n \right)
$$ 
and rewrite the equation as 
$$
\partial_x\hat{U}=s\bar{M}(\xi',\eta',1/s)\hat{U}.
$$
Let $\bar{R}_M^S=\bar{R}_M^S(\xi',\eta',1/s)$ and $\bar{R}_M^U=\bar{R}_M^U(\xi',\eta',1/s)$ be the respective right-stable and right-unstable matrices of $\bar{M}=\bar{M}(\xi',\eta',1/s)$:
\begin{align*}
\bar{M}\bar{R}_M^S=\bar{R}_M^S\bar{M}^S,\qquad\qquad 
\bar{M}\bar{R}_M^U=\bar{R}_M^U\bar{M}^U,
\end{align*}
where $\bar{M}^S$ is a stable-matrix and $\bar{M}^U$ is an unstable-matrix. 
In view of the Schur decomposition, we may choose $\bar{R}_M^S$ and $\bar{R}_M^U$ such that
$\bar{R}_M^{S}\bar{R}_M^{S*} +\bar{R}_M^{U}\bar{R}_M^{U*} = I$. 
Then from \eqref{5.16} we obtain
\begin{align*}
\left({\begin{array}{*{20}c}
  \vspace{1.5mm} \bar{R}_M^{S*} \\
                 \bar{R}_M^{U*}
\end{array}}\right)\partial_{x}\hat{U}= 
  \left({\begin{array}{*{20}c}
  \vspace{1.5mm}\bar{M}^S & \\
                    & \bar{M}^U
\end{array}}\right)\left({\begin{array}{*{20}c}
  \vspace{1.5mm} \bar{R}_M^{S*} \\
                 \bar{R}_M^{U*}
\end{array}}\right)\hat{U}.
\end{align*}
Since $\|\hat{U}(\cdot,\xi)\|_{L^2(\mathbb{R}^+)}<\infty$ for a.e. $\xi$ and $\bar{M}^U$ is an unstable-matrix, it must be $\bar{R}_M^{U*}\hat{U}=0$ and the boundary condition in \eqref{5.16} becomes
\begin{align*}
B\hat{U}(0,\xi)= B\bar{R}_M^{S}\bar{R}_M^{S*}\hat{U}(0,\xi)=\hat{b}(\xi). 
\end{align*}

Recall the definition of $M(\xi,\eta)$ in the GKC, we know that $M(\xi',\eta')=\bar{M}(\xi',\eta',0)$. According to the GKC, we have 
$$
|\det\{B\bar{R}_M^S(\xi',\eta',0)\}|\geq c_K.
$$ 
Note that the right-stable matrix $\bar{R}_M^S(\xi',\eta',1/s)$ is continuous with respect to $1/s$ at $1/s=0$. 
For sufficiently small $\epsilon$, we know that $1/s$ is also sufficiently small and thereby
$$
|\det\{B\bar{R}_M^S(\xi',\eta',1/s)\}|\geq \bar{c}_K
$$ 
with $\bar{c}_K$ a constant for sufficiently small $\epsilon<\epsilon_0\ll1$. 
This means the matrix $(B\bar{R}_M^S)^{-1}$ is uniformly bounded and thereby
\begin{align*}
\left|\hat{U}(0,\xi)\right|^2=& \left|\bar{R}_M^{S}\bar{R}_M^{S*}\hat{U}(0,\xi)\right|^2 
=\left|\bar{R}_M^{S}(B\bar{R}_M^S)^{-1}\hat{b}(\xi)\right|^2
\leq C|\hat{b}(\xi)|^2.
\end{align*}
By Parseval's identity, the last inequality leads to 
\begin{align*} 
 \int_0^{\infty}e^{-2t Re\xi }|U(0,t)|^2dt 
\leq C \int_0^{\infty}e^{-2t Re\xi }\left|b(t)\right|^2dt 
\leq C \int_0^{\infty} \left|b(t)\right|^2dt \quad \text{for}~~Re\xi>0.
\end{align*}
Because the right-hand side is independent of $Re\xi$, we have 
$$
\int_0^{\infty} |U(0,t)|^2dt \leq C \int_0^{\infty} \left|b(t)\right|^2dt  .
$$
By using the trick from \cite{GKO}, the integral interval $[0,\infty)$ in the last inequality can be changed to $[0,T]$.

At last, we multiply the equation with $U^T$ from the left to get
\begin{align*}
\frac{d}{dt}|U|^2+ \partial_{x}(U^TAU)\leq U^T\bar{Q}U.
\end{align*}
Integrating the last inequality over $x\in[0,+\infty)$ and using Gronwall's inequality yields
\begin{align*}
\max_{t\in[0,T]}\|U(\cdot,t)\|_{L^2(\mathbb{R}^+)}^2\leq & ~ C(T)\int_0^{T} |U(0,t)|^2dt \leq C(T) \int_0^{T} \left|b(t)\right|^2dt.
\end{align*} 
\end{proof}

\end{appendices}

\bibliographystyle{abbrv}
\bibliography{ref}

\end{document}